\documentclass{siamart1116}

\usepackage{lipsum}
\usepackage{amsfonts}
\usepackage{graphicx}
\usepackage{epstopdf}
\usepackage{algorithmic}
\usepackage{color}

\ifpdf
  \DeclareGraphicsExtensions{.eps,.pdf,.png,.jpg}
\else
  \DeclareGraphicsExtensions{.eps}
\fi

\numberwithin{theorem}{section}

\newcommand{\TheTitle}{Globally Exact Asymptotics for Integrals with Arbitrary Order Saddles} 
\newcommand{\TheAuthors}{T. Bennett, C. J. Howls, G. Nemes,  A. B. Olde Daalhuis}

\headers{Globally Exact Asymptotics for Integrals}{\TheAuthors}

\title{{\TheTitle}\thanks{Submitted to the editors October 27, 2017.
\funding{G. Nemes and A. B. Olde Daalhuis were supported by a research grant (GRANT11863412/70NANB15H221) from the National Institute of Standards and Technology.  T. Bennett was sponsored by an EPSRC studentship.}}}

\author{
  T. Bennett\thanks{SMSAS, University of Kent, Sibson Building, Parkwood Road, Canterbury CT2 7FS, UK.
    }
  \and
  C. J. Howls\thanks{Mathematical Sciences, University of Southampton, Highfield, Southampton SO17 1BJ, UK.}
 \and
 G. Nemes\thanks{School of Mathematics, University of Edinburgh, James Clerk Maxwell Building, The King's Buildings, Edinburgh EH9 3FD, UK.}
  \and
  A. B. Olde Daalhuis\footnotemark[4]
}

\usepackage{amsopn}

\usepackage{mathrsfs}
\usepackage{subfigure}

\ifpdf
\hypersetup{
  pdftitle={\TheTitle},
  pdfauthor={\TheAuthors}
}
\fi

\def\HypTermZero#1{{\bf F}^{(0)}(#1)}
\newcommand{\HypTermOne}[4]{{{\bf F}^{(1)} \left( #1;\begin{array}{c}{#2}\\ {#3}\\ {#4}\\ \end{array} \right)}}
\newcommand{\HypTermTwo}[7]{{{\bf F}^{(2)} \left( #1;\begin{array}{c}{#2},\\ {#4},\\ {#6},\\ \end{array} \begin{array}{c}{#3}\\ {#5}\\ {#7}\\ \end{array}\right)}}
\newcommand{\HypTermK}[3]{{{\bf F}^{(#1)} \left( #2;#3 \right)}}
\def\i{\mathrm{i}}
\def\d{\mathrm{d}}
\def\e{\mathrm{e}}
\def\ifrac#1#2{\textstyle{{#1}\over{#2}}\displaystyle}

\begin{document}

\maketitle

\begin{abstract}
We derive the first exact, rigorous but practical, globally valid remainder terms for asymptotic expansions about saddles and contour endpoints of arbitrary order degeneracy derived from the method of steepest descents. The exact remainder terms lead naturally to sharper novel asymptotic bounds for truncated expansions that are a significant improvement over the previous best existing bounds for quadratic saddles derived two decades ago. We also develop a comprehensive hyperasymptotic theory, whereby the remainder terms are iteratively re-expanded about adjacent saddle points to achieve better-than-exponential accuracy. By necessity of the degeneracy, the form of the hyperasymptotic expansions are more complicated than in the case of quadratic endpoints and saddles, and require generalisations of the hyperterminants derived in those cases. However we provide efficient methods to evaluate them, and we remove all possible ambiguities in their definition.  We illustrate this approach for three different examples, providing all the necessary information for the practical implementation of the method.
\end{abstract}



\section{Introduction}

From catastrophe theory it is well known that integrals with saddle points may be used to 
compactly encapsulate the local behaviour of linear wavefields near the underlying 
organising caustics, see for example \cite{PS96,Berry80}. The saddle points correspond to rays of the 
underpinning ODEs or PDEs.  Their coalescence corresponds to tangencies of the rays at the caustics, leading to nearby peaks in the wave 
amplitude.  On the caustics, the coalesced saddle points are degenerate.  The local analytical
behaviour on the caustic may be derived from an asymptotic expansion about the degenerate saddle \cite[\href{http://dlmf.nist.gov/36}{Ch.~36}]{NIST:DLMF}.  An analytical understanding of the 
asymptotic expansions involving degenerate saddles is thus essential to an examination of the wavefield behaviour on caustics.  A modern 
approach to this includes the derivation of globally exact remainders, sharp error bounds and the exponential improvement of the 
expansions to take into account the contributions of terms beyond all orders.  

Recent work in quantum field and string theories, e.g.,  \cite{Dunne16a, Dunne16b, Dunne16c, Dunne17a}  
has led to a major increase in interest in such resurgent approaches in the context of integral asymptotics.
One of the reasons to study the higher orders of expansions in QFT is that they can reveal information as to the location of remote critical points, 
corresponding to physical quantities that can give rise to non-perturbative effects.��
A notable recent success is that high order calculations can be used to uncover previously unknown functional relationships between perturbative and non-perturbative effects within quantum eigenvalue expansions
\cite{Dunne14}.

The first globally exact remainders for asymptotic expansions of integrals possessing simple saddle points were derived by Berry and Howls \cite{BH91}. The remainder terms were expressed in terms of self-similar integrals over doubly infinite contours passing through a set of adjacent simple saddles. Boyd \cite{Boyd93} provided a rigorous justification of the exact remainder terms, together with significantly improved error bounds.

The remainder terms automatically incorporated and precisely accounted for the Stokes phenomenon \cite{Stokes1857}, whereby exponentially subdominant asymptotic contributions are switched on 
as asymptotics or other parametric changes cause the contour of integration to deform to pass through the adjacent saddles. The Stokes phenomenon occurs across subsets in parameter space called Stokes lines. 

Re-expansion of the exact remainder term about the adjacent saddles, using their own exact remainder terms led to a hyperasymptotic 
expansion, which delivered better-than-exponential numerical accuracy.

Subsequent work extended globally exact remainder terms and hyperasymptotic analysis to integrals over contours with finite endpoints \cite{Howls92} 
and multiple integrals \cite{Howls97}, \cite{DH02}. Parallel approaches to differential equations using Cauchy--Heine and Borel transforms were taken by 
Olde Daalhuis and Olver \cite{OO95b}, \cite{OD98b}. This resulted in efficient methods for computation of the universal hyperterminants 
\cite{OD98c}.   The efficient computation of hyperterminants not only made hyperasymptotic expansions numerically feasible, but more 
importantly, in the absence of the geometric information present in single dimensional integral calculations, allowed them to be used to 
calculate the Stokes constants that are required in an exponentially accurate asymptotic calculation involving, for example, the solution 
satisfying given boundary data.

However, the general case of globally exact remainder terms and hyperasymptotic expansions of a single-dimensional integral possessing a 
set of arbitrary order degenerate saddle points has not yet been considered. The purpose of this paper is to fill this surprising gap.

Hence, in this paper, we provide the first comprehensive globally exact asymptotic theory for integrals with analytic integrands involving finite 
numbers of arbitrarily degenerate saddle points. It incorporates the special case of Berry and Howls \cite{BH91} and Howls \cite{Howls92}. 
However the complexity of the situation uncovers several new features that were not present in the simple saddle case.

First, the nature of the steepest paths emerging from degenerate saddles gives multiple choices as to which contours might be integrated 
over, or which might contribute to the remainder term.  It is necessary to adopt a more stricter convention regarding the choice of steepest 
paths to clarify the precise nature of the contributions to the remainder and hyperasymptotic expansions.

Second, the degenerate nature requires us to explore additional Riemann sheets associated to the local mappings about the saddle points. 
This gives rise to additional complex phases, not obviously present in the simple saddle case, that must be taken into account depending on 
the relative geometrical disposition of the contours.

Third, we provide sharp, rigorous bounds for the remainder terms in the Poincar\'e asymptotic expansions of integrals with arbitrary critical points.
In particular, we improve the results of Boyd \cite{Boyd93} who considered integrals with only simple saddles. Our bounds are sharper, and have
larger regions of validity.

Fourth, the hyperasymptotic tree structure that underpins the exponential improvements in accuracy is {\it prima facie} more complicated.  At 
the first re-expansion of a remainder term, for each adjacent degenerate saddle there are two contributions arising from the choice of 
contour over which the remainder may be taken.  At the second re-expansion, each of these two contributions may give rise to another two, 
and so on.  Hence, while the role of the adjacency of saddles remains the same, the numbers of terms required at each hyperasymptotic 
level increases twofold for each degenerate saddle at each level. Fortunately these terms may be related, and so the propagation of 
computational complexity is controllable.

Fifth, the hyperterminants in the expansion are more complicated than those in \cite{BH91}, \cite{Murphy97}, \cite{OD98b} or  \cite{OD98c}. However we provide efficient methods to evaluate them.

Sixth, the results of this integral analysis reveals new insights into the asymptotic expansions of higher order differential equations.  

There have been several near misses at a globally exact remainder term for degenerate saddles arising from single dimensional integrals.  

Ideas similar to those employed by Berry and Howls were used earlier by Meijer. In a series of papers \cite{Meijer32a}, \cite{Meijer32b}, 
\cite{Meijer33} he derived exact remainder terms and realistic error bounds for specific special functions, namely Bessel, Hankel and Anger--Weber-type functions. Nevertheless, he missed the extra step that would have led him to more general remainder terms of \cite{BH91}.

Dingle \cite{Dingle73}, whose pioneering view of resurgence underpins most of this work, considered expansions around cubic saddle 
points, and gave formal expressions for the higher order terms.  However, he did not provide exact remainder terms or consequent (rigorous) 
error estimates.

Berry and Howls, \cite{BH93}, \cite{BH94}, considered the cases of exponentially improved uniform expansions of single dimensional 
integrals as saddle points coalesced. The analysis \cite{BH93} focused on the form of the late terms in the more complicated uniform 
expansions. They \cite{BH94} provided an approximation to the exact remainder term between a simple and an adjacent cluster of saddles 
illustrating the persistence of the error function smoothing of the Stokes phenomenon \cite{Berry89} as the Stokes line was crossed.  Neither 
of these works gave globally exact expressions for remainder terms involving coalesced, degenerate saddles.  

Olde Daalhuis \cite{OD00a} considered a Borel plane treatment of uniform expansions, but did not extend the work to include arbitrary degenerate saddles.

Breen \cite{Breen99} briefly considered the situation of degenerate saddles. The work restricted attention to cubic saddles and, like all the above work, did not provide rigorous error bounds or develop a hyperasymptotic expansion. 

Other notable work dealing with exponential asymptotics include \cite{Boyd98},
\cite{Boyd99},
\cite{Grimshaw10},
\cite{Jones97},
\cite{Lombardi00},
\cite{Segur91}, and
\cite{Ward10}.

It should be stressed that the purpose of a hyperasymptotic approach is not {\it per se} to calculate functions to high degrees of numerical 
accuracy: there are alternative computational methods. Rather, hyperasymptotics is as an analytical tool to incorporate 
exponentially small contributions into asymptotic approximations, so as to widen the domain of validity, understand better the underpinning 
singularity structures and to compute invariants of the system such as Stokes constants whose values are often assumed or left as 
unknowns by other methods (see for example \cite{Gil07}).

The idea for this paper emerged from the recent complementary and independent thesis 
work of \cite{Bennett15}, \cite{Nemes15}, which gave rise to the current collaboration.  
This collaboration has resulted in the present work which incorporates not only a 
hyperasymptotic theory for both expansions arising from non-degenerate and degenerate 
saddle points, but also significantly improved rigorous and sharp error bounds for the 
progenitor asymptotic expansions.

The structure of the paper is as follows.

In Section \ref{sec2}, we introduce arbitrary finite integer degenerate saddle points.  In Section \ref{sec3}, we derive the exact remainder 
term for an expansion about a semi-infinite steepest descent contour emerging from a degenerate saddle and running to a valley at infinity.  
The remainder term is expressed as a sum of terms of contributions from other, adjacent saddle points of the integrand.  Each of these 
contributions is formed from the difference of two integrals over certain semi-infinite steepest descent contours emerging from the adjacent 
saddles.

In Section \ref{sec4}, we iterate these exact remainder terms to develop a hyperasymptotic expansion. We introduce novel hyperterminants (which simplify to those of Olde Daalhuis  \cite{OD98c} when the saddles are non-degenerate).  

In Section \ref{sec5}, we provide explicit rigorous error bounds for the zeroth hyperasymptotic level. These novel bounds are sharper than those derived by Boyd \cite{Boyd93}.

In Section \ref{sec6}, we illustrate the degenerate hyperasymptotic method with an application to an integral related to the Pearcey function, 
evaluated on its cusp caustic. The example involves a simple and doubly degenerate saddle. In Section \ref{sec7}, we provide an illustration 
of the extra complexities of a hyperasymptotic treatment of degeneracies with an application to an integral possessing triply and quintuply 
degenerate saddle points. In this example, we also illustrate the increased size of the remainder near a Stokes line as predicted in Section 
\ref{sec5}.  In Section \ref{sec8}, we give an example of how it is possible to make an algebraic (rather than geometric) determination of the saddles 
that contribute to the exact remainder terms in a swallowtail-type integral through a hyperasymptotic examination of the late terms in the 
saddle point expansion.

In Section \ref{sec9}, we conclude with a discussion on the application of the results of this paper to the (hyper-) asymptotic expansions of higher order differential equations.

\section{Definitions and assumptions}\label{sec2}


Let $\omega_j$ be a positive integer, with $j=1,2,\dots$ an integer index. Consider a function $f(t)$, analytic in a domain of the complex plane. The point $t^{(j)}$, is called a critical point of order $\omega_{j}-1$ of $f(t)$, if
\begin{equation}\nonumber
f^{(p)} (t^{(j)} ) =  0 \ \ \ \text{but} \ \ f^{(\omega _j )} (t^{(j)} ) \ne 0, \ \
\text{for \ all} \ \ p=1,\ldots, \omega_{j}-1. 
\end{equation}
When $\omega_{j}=1,2, >2$, $t^{(j)}$ is, respectively, a linear endpoint, a simple saddle point, a degenerate saddle point.  For analytic $f(t)$, 
the saddle points are then all isolated.  Henceforth we denote the value of $f(t)$ at $t=t^{(j)}$ by $f_{j}$.

We shall derive the steepest descent expansion, together with its exact remainder term, of integrals of the type
\begin{equation}\label{eq1}
I^{(n)}( z; \alpha_{n} ) = \int_{\mathscr{P}^{\left(n\right)}} \e^{ - zf(t)} g(t)\d t , \ \  z=|z|\e^{\i\theta}, \ \ |z|\rightarrow \infty,
\end{equation}
where $\mathscr{P}^{(n)}=\mathscr{P}^{(n)}( \theta ; \alpha_{n} )$ is one of the $\omega_{n}$ paths of steepest descent emanating from the 
$(\omega_{n}-1)^{\rm st}$-order critical point $t^{(n)}$ of $f(t)$ and passing to infinity in a valley of $\mathop{\rm Re}\left[-\e^{\i\theta } (f(t) - f_n)\right]
$. 

Suppose we use the notation of ($\omega_{n}\rightarrow {\bf \omega_{m}}$) to indicate the remainder term that rises from an asymptotic 
expansion about a endpoint/saddle point $n$ of order $\omega_{n}$ in terms of the adjacent (in a sense to be defined later) set of saddles 
${\bf m}=\{m_{1},m_{2},m_{3},\dots\}$, of orders corresponding to the values $ \omega_{{\bf m}}=\{\omega_{m_{1}},\omega_{m_{2}},\dots \}
$.  Thus Berry and Howls \cite{BH91} dealt with $(\omega_{n}\rightarrow \omega_{{\bf m}}) = (2\rightarrow {\bf 2})$, for doubly infinite contours.  Howls \cite{Howls92} 
dealt with $(1\rightarrow {\bf 2})$ and the $(2\rightarrow {\bf 2})$.  
Our goal here is to derive the exact remainder terms for arbitrary 
integers ($\omega_{n}\rightarrow {\bf \omega_{m}}$).

On the steepest path $\mathscr{P}^{\left(n\right)}( \theta ; \alpha_{n} )$ emerging from $t^{(n)}$, we have
\begin{equation}\label{eq1a}
\arg\left[\e^{\i\theta } (f(t) - f_{n})\right]=2\pi \alpha_{n}, 
\end{equation}
for a suitable integer $\alpha_{n}$ (see Figure \ref{fig1}). 

The local behaviour of $f(t)$ at the critical point $t^{\left(n \right)}$ of order $\omega_{j}-1$ is given by
\begin{equation}\label{eq4a}
f(t) - f_n = \frac{f^{(\omega _n )} (t^{(n)} )}{\omega _n !}\left(t - t^{(n)} \right)^{\omega _n }  + \mathcal{O}\left(\left|t - t^{\left( n \right)} \right|^{\omega_{n}+1}\right).
\end{equation}

From \eqref{eq1a} and \eqref{eq4a}, we hence find that
\begin{equation}\label{alpha}
\alpha _n  = \frac{\theta  + \arg (f^{(\omega _n )} (t^{(n)} )) + \omega_n\varphi}{2\pi},
\end{equation}
where $-\pi  < \arg \left(f^{(\omega _n )} (t^{(n)} )\right) \le \pi$, and $\varphi$ ($-\pi < \varphi \le \pi$) is the angle of the slope of $\mathscr{P}^{(n)}( \theta ; \alpha_{n} )$ at $t^{(n)}$, i.e., $\lim \left(\arg(t-t^{(n)})\right)$ as $t\to t^{(n)}$ along $\mathscr{P}^{(n)}( \theta ; \alpha_{n} )$.
\begin{figure}
\centering\includegraphics[width=0.6\textwidth]{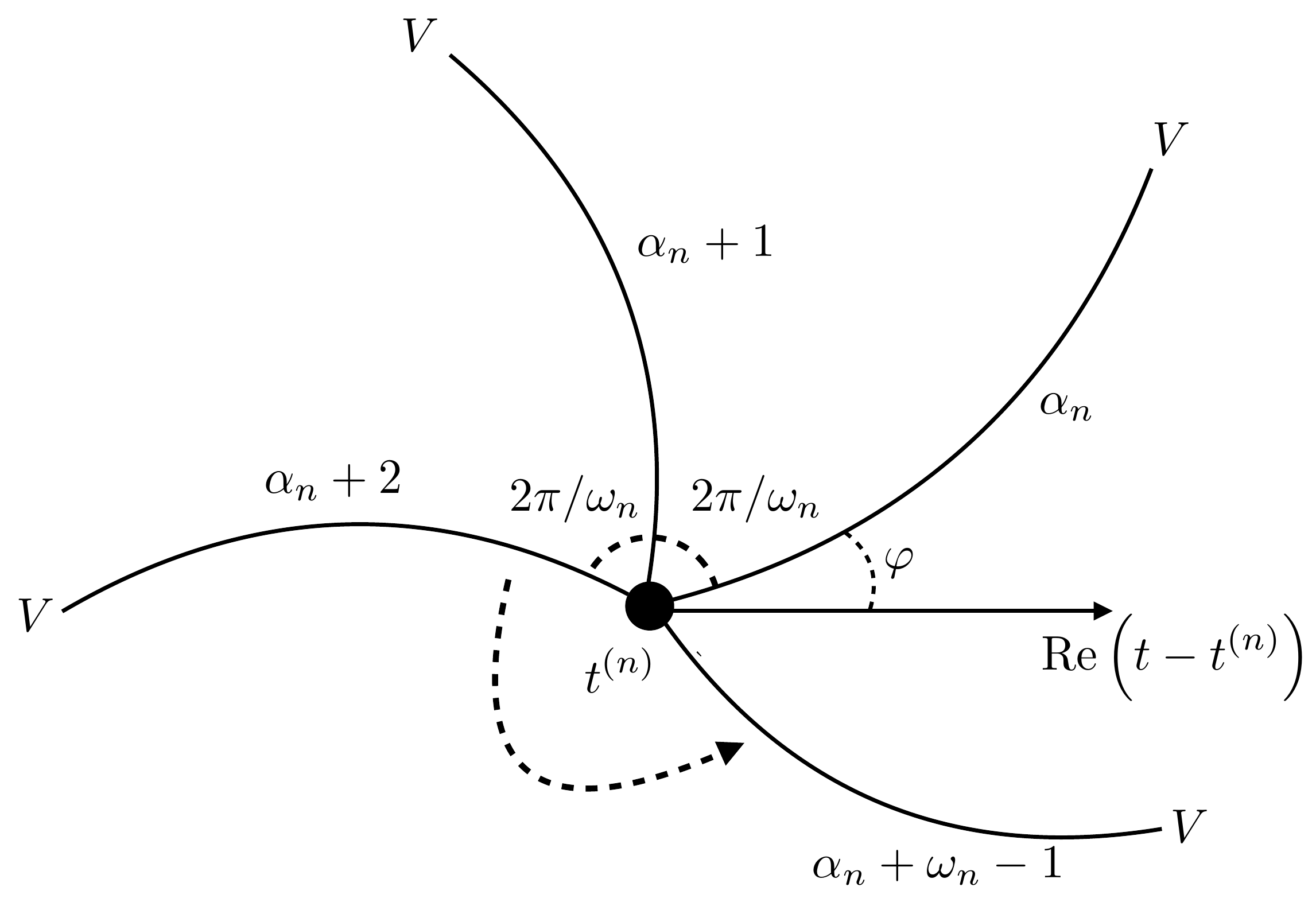}%
\caption{The $\omega_{n}$ paths of steepest descent emanating from the $(\omega_{n}-1)^{{\rm st}}$-order critical point $t^{(n)}$ of $f(t)$.}
\label{fig1}
\end{figure}

The functions $f(t)$ and $g(t)$ are assumed to be analytic in the closure of a domain $\Delta^{(n)}$. We suppose further that $\left| {f(t)} \right| \to \infty$ as $t\to \infty$ in $\Delta^{(n)}$, and $f(t)$ has several other saddle points in the complex $t$-plane at $t=t^{\left(j\right)}$ labelled by $j \in \mathbb{N}$. 


The domain $\Delta^{(n)}$ is defined by considering all the steepest descent paths for different values of $\theta$, which emerge from 
the critical point $t^{(n)}$. In general these paths can end either at infinity or at a singularity of $f(t)$. We assume that all of them end at infinity. Since there are no branch points of $f(t)$ along these paths, any point in the $t$-plane either cannot be reached by any path of steepest descent issuing from $t^{(n)}$, or else by only one. A continuity argument shows that the set of all the points which can be reached by a steepest descent path from $t^{(n)}$ forms the closure of a domain in the $t$-plane. It is this domain which we denote by $\Delta^{(n)}$, see for example Figure \ref{fig2}.

Instead of considering the raw integral \eqref{eq1}, it will be convenient to consider instead its slowly varying part, defined by
\begin{equation}\label{eq2}
T^{(n)} ( z ;\alpha_{n}) := \omega_{n}z^{1/\omega_{n}}\e^{zf_{n}} I^{(n)} ( z ;\alpha_{n}) =  \omega_{n}z^{1/\omega_{n}}\int_{\mathscr{P}^{\left(n \right)} } \e^{ - z(f(t) - f_{n})} g(t)\d t.
\end{equation}
The $\omega_{n}^{\rm th}$ root is defined to be positive on the positive real line and is defined by analytic continuation elsewhere. 
We call $T^{(n)} ( z ;\alpha_{n})$ the slowly varying part because it is $\mathcal{O}(1)$ as $z\to\infty$ (cf.~\eqref{FirstExpansion}).

It is convenient to introduce the following notation for the special double integrals and their coefficients in the asymptotic expansions
\begin{gather}
\begin{split}\label{double}{\bf T}^{(m)}(u;\alpha_{n})&=T^{(m)}(u;\alpha_{n})-T^{(m)}(u;\alpha_{n}+1),\\
\qquad {\bf T}^{(m)}_r(\alpha_{n})&=T^{(m)}_r(\alpha_{n})-T^{(m)}_r(\alpha_{n}+1).
\end{split}
\end{gather}

The path $\mathscr{P}^{\left( n \right)} ( \theta ;\alpha_{n} )$ passes through certain other saddle points $t^{\left(m\right)}$ when $\theta  = 
\theta _{nm}^{[1]} ,\theta _{nm}^{[2]} ,\theta _{nm}^{[3]} , \ldots$,  with $\theta _{nm}^{[j]}  = \theta _{nm}^{[k]} \bmod 2\pi \omega _n$. Such 
saddle points are defined as being ``adjacent'' to $t^{(n)}$. 

Initially we chose the value of $\theta$ so that the steepest descent path $\mathscr{P}^{(n)}( \theta ;
\alpha_{n})$ in \eqref{eq1} does not encounter any of the saddle points of $f(t)$ other than $t^{\left(n\right)}$. We define
\begin{equation}\nonumber
\theta _{nm}^ +   := \min \left\{ \theta _{nm}^{[j]} :j \ge 1,\theta  < \theta _{nm}^{[j]} \right\} \;\; \text{ and } \;\;
\theta _{nm}^ -   := \max \left\{ \theta _{nm}^{[j]} :j \ge 1,\theta _{nm}^{[j]}  < \theta  \right\}.
\end{equation}
Note that $\theta _{nm}^ + = \theta _{nm}^ - +2\pi\omega_n$. Thus, in particular, $\theta$ is restricted to an interval
\begin{equation}\label{eq3}
\theta_{nm_1}^ - < \theta  < \theta_{nm_2}^ +,
\end{equation}
where $\theta_{nm_1}^- := \max _m \theta _{nm}^-$ and $\theta_{nm_2 }^+  := \min _m \theta_{nm}^+$. We shall suppose that $f(t)$ and $g(t)$ grow sufficiently rapidly at infinity so that the integral \eqref{eq1} converges for all values of $\theta$ in the interval \eqref{eq3}.

Let $\Gamma^{(n)}=\Gamma^{(n)}( \theta )$ be an infinite contour that encircles the path $\mathscr{P}^{(n)}( \theta ; \alpha_{n})$ in the positive direction within $\Delta^{(n)}$ (see Figure \ref{fig2}(a)).
This contour $\Gamma ^{(n)} ( \theta  )$ is now deformed by expanding it onto the boundary of $\Delta^{(n)}$. We assume that the set of saddle points which are adjacent to $t^{(n)}$ is non-empty and finite. Under this assumption, it is shown in Appendix \ref{appC} that the boundary of $\Delta^{(n)}$ can be written as a union of contours $\bigcup\nolimits_m \mathscr{P}^{(m)} (\theta _{nm}^ +  ,\alpha _{nm}^ +  ) \cup -\mathscr{P}^{(m)} (\theta _{nm}^ -  ,\alpha _{nm}^ -  )$, where $\mathscr{P}^{(m)} (\theta _{nm}^ \pm  ,\alpha _{nm}^ \pm  )$ are steepest descent paths emerging from the adjacent saddle $t^{(m)}$ (see Figure \ref{fig2}(b)). These paths are called the adjacent contours. The integers $\alpha _{nm}^ \pm $ are computed analogously to $\alpha_n$ (cf. \eqref{alpha}) as
\begin{equation}\label{alphanm}
\alpha _{nm}^ \pm   = \frac{\theta _{nm}^ \pm   + \arg (f^{(\omega _m )} (t^{(m)} )) + \omega _m \varphi ^ \pm  }{2\pi },
\end{equation}
where $-\pi  < \arg (f^{(\omega _m )} (t^{(m)} )) \le \pi$, and $\varphi^\pm$ ($-\pi < \varphi^\pm \le \pi$) is the angle of the slope of $\mathscr{P}^{(m)} (\theta _{nm}^ \pm  ,\alpha _{nm}^ \pm  )$ at the $(\omega_m-1)^{{\rm st}}$-order saddle point $t^{(m)}$ to the positive real axis. We assume initially that each adjacent contour contains only one saddle point.\footnote{This condition may be relaxed by extending the definition of integrals of the form \eqref{eq2} to include the limiting case when the steepest descents path connects to other saddle points. Also, a limiting case, such as \eqref{hyptermlimit}, has to be used for the generalised hyperterminants in the corresponding re-expansions.} The other steepest descent paths from $t^{(m)}$ are always external to the domain $\Delta^{(n)}$.

Finally, we introduce the so-called singulants $\mathcal{F}_{nm}^\pm$ (originally defined by Dingle \cite[pp. 147--149]{Dingle73}) via
\begin{equation}\nonumber
\mathcal{F}_{nm}^\pm := |f_{m} - f_{n}|\e^{\i \arg \mathcal{F}_{nm}^\pm},\quad \arg \mathcal{F}_{nm}^\pm  = - \theta_{nm}^\pm + 2\pi \alpha _n.
\end{equation}

\begin{figure}
\hfill
\subfigure[ \hspace{1mm} ]{\includegraphics[width=0.40\hsize]{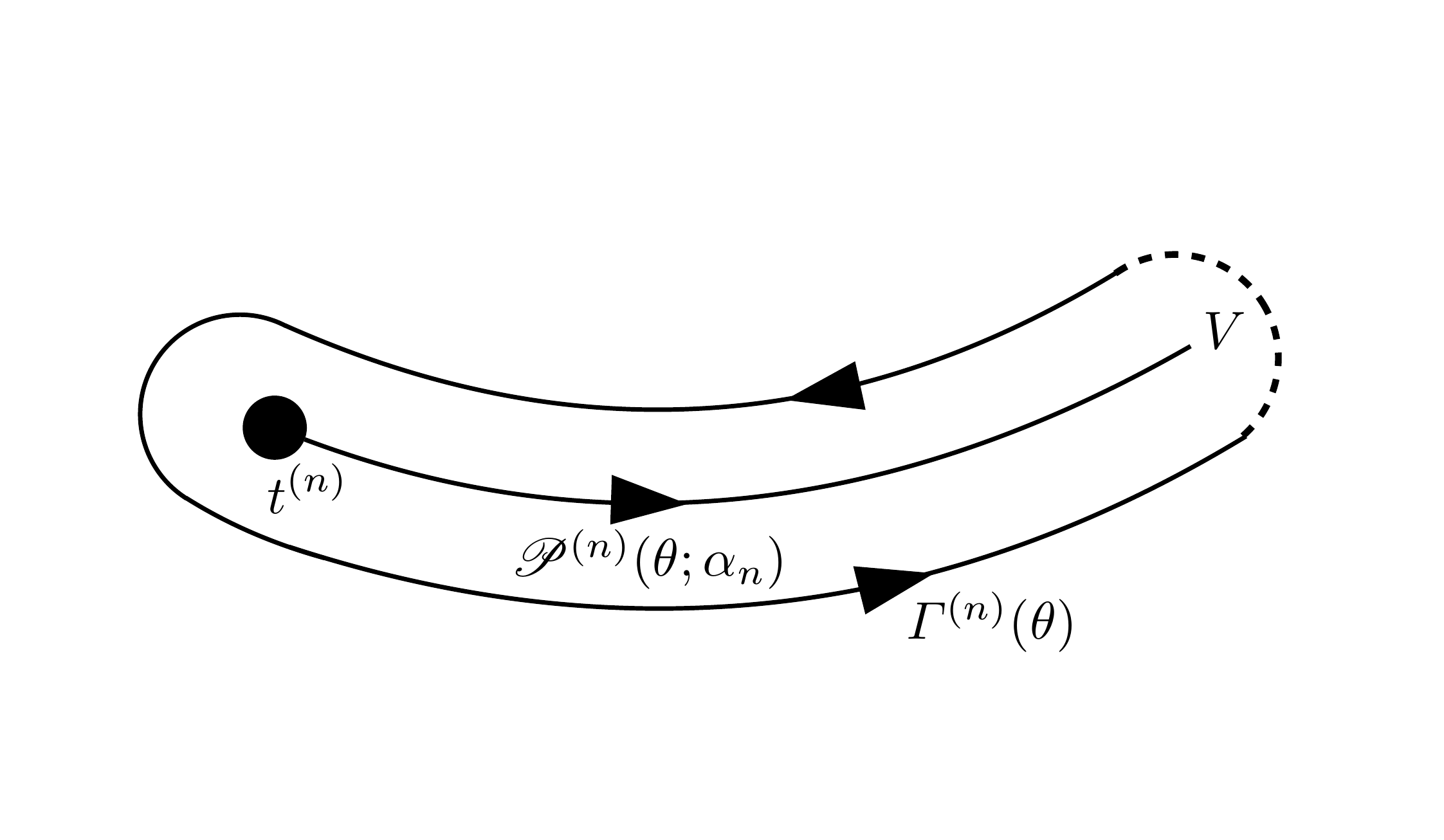}}
\hfill
\subfigure[\hspace{1mm} ]{\includegraphics[width=0.55\hsize]{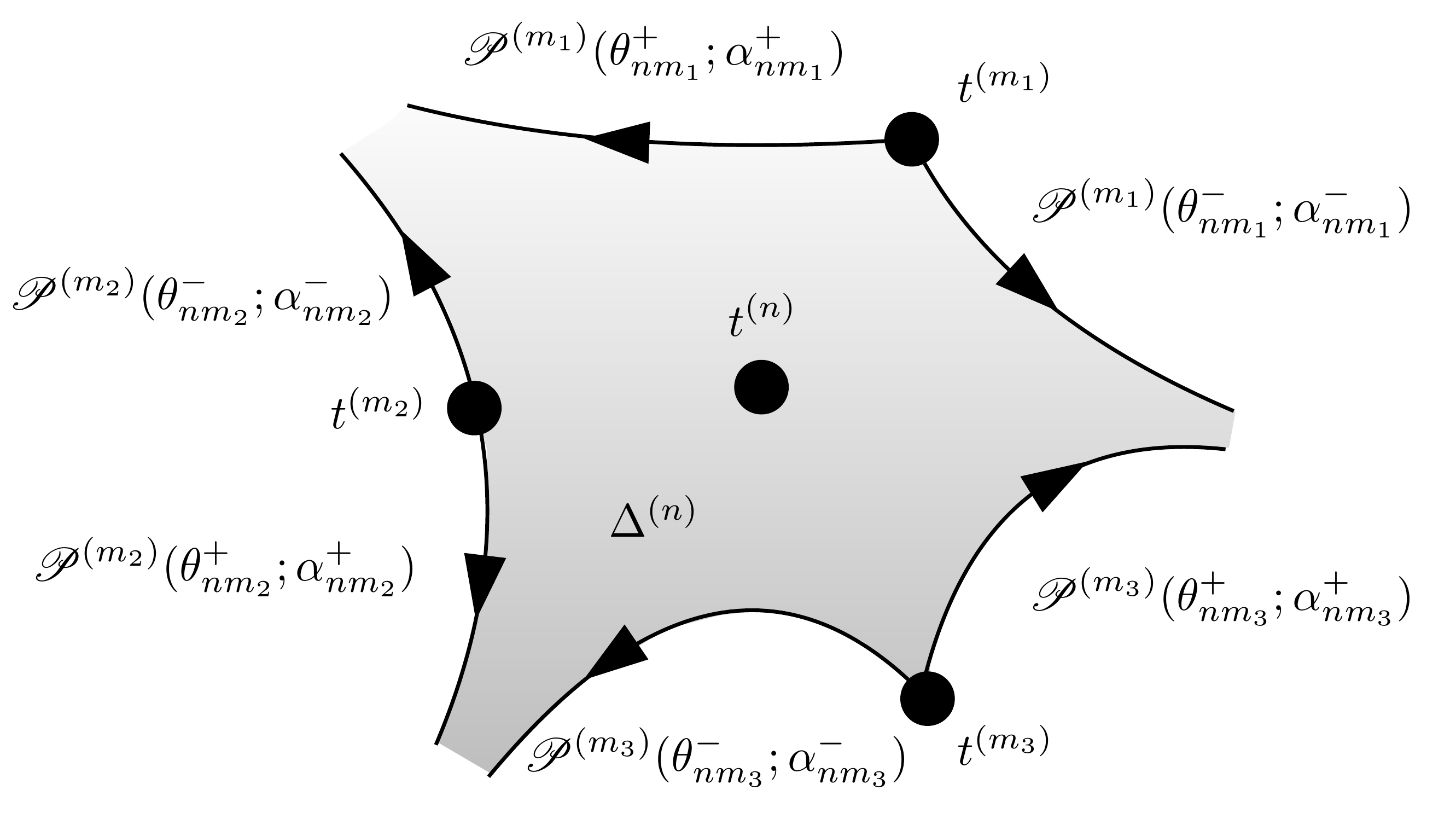}}
\hfill
\caption{Contours used in the derivation of the exact remainder terms. (a) The contour $\Gamma^{(n)}(\theta)$ relative to the integration contour $\mathscr{P}^{(n)}(\theta;\alpha_{n})$ as used in \eqref{eq6}. (b) A schematic representation of the saddle points $t^{(m_{j})}$ that are adjacent to $t^{(n)}$ and the adjacent contours $\mathscr{P}^{(m_j)}$ emanating from them in \eqref{eq12}, together with the domain $\Delta^{(n)}$.}
\label{fig2}
\end{figure}

\section{Derivation of exact remainder term}\label{sec3}

In this section we will show that
\begin{equation}\label{FirstExpansion}
T^{(n)} (z ;\alpha_{n})= \sum\limits_{r = 0}^{N - 1} \frac{T_r^{(n)}(\alpha_{n})}{z^{r/\omega_{n}} }  + R_N^{(n)} (z;\alpha_{n} ),
\end{equation}
where
\begin{align}
T_r^{(n)} (\alpha_{n}) = \; & \e^{\frac{2\pi \i \alpha_{n}(r+1)}{\omega_{n}}}  \frac{\Gamma \left( \frac{r + 1}{\omega_{n}} \right)}{2\pi \i}\oint_{t^{(n)}} \frac{g(t)}{\left(f(t) -f_{n}\right)^{(r + 1)/\omega_{n} }}\d t \label{eq101} \\  
= \; & \e^{\frac{2\pi \i \alpha_{n}(r+1)}{\omega_{n}}}
\left( \frac{\omega _n !}{f^{(\omega _n )} (t^{(n)} )} \right)^{(r + 1)/\omega _n } \frac{\Gamma \left( \frac{r + 1}{\omega_{n}}\right)}{\Gamma \left( r + 1\right)} \nonumber \\ & \times \left[ {\frac{{\d^r }}{{\d t^r }}\left( {g(t)\left( \frac{f^{(\omega _n )} (t^{(n)} )}{\omega _n !}\frac{\left(t - t^{\left( n \right)} \right)^{\omega_{n} }}{f(t) - f_{n}} \right)^{(r + 1)/\omega_{n} }} \right)} \right]_{t = t^{(n)} } \label{eq11},
\end{align}
and for the remainder we have
\begin{equation}\label{Step1}
R_{N}^{(n)}(z;\alpha_n)=\sum_{m(n)}\frac{z^{(1-N)/\omega_n}}{2\pi\i\omega_m} 
\int_0^{\infty\e^{\i\theta_{nm}^{+}}}
\frac{\e^{-\mathcal{F}_{nm}^+ u} u^{\frac{N+1}{\omega_n}-\frac{1}{\omega_m}-1}}{z^{1/\omega_n}-u^{1/\omega_n}}
{\bf T}^{(m)}(u;\alpha_{nm}^{+})\d u,
\end{equation}
in which $m(n)$ means that we sum over all saddles that are adjacent to $n$.
The result (\ref{Step1}) for the exact remainder term of the asymptotic expansion around the degenerate saddle $t^{(n)}$, 
expressed in terms of the adjacent (other degenerate) saddles $t^{(m)}$, is one of the main results of this paper.  

If we omit the remainder term $R_N^{(n)} (z;\alpha_{n} )$ in \eqref{FirstExpansion} and formally extend the sum to infinity, the result becomes the asymptotic expansion of an integral with $(\omega_n-1)^{\rm st}$-order endpoint (cf. \cite[eq. (1.2.16), p. 12]{Paris11}). A representation equivalent to \eqref{eq101} was given, for example, by Copson \cite[p. 69]{Copson65}. The expression \eqref{eq11} is a special case of Perron's formula (see, e.g., \cite{Nemes13}).

In the examples below we use \eqref{eq11} to compute conveniently and analytically the exact coefficients. However, we remark that \eqref{eq101} may be combined with the trapezoidal rule evaluated at periodic points on the loop contour about $t^{(n)}$ (see for example
\cite{TW14}) to give an efficient approximation for the coefficients as
\begin{equation}\label{trapezoidal}
T_r^{(n)} (\alpha_{n}) \approx  \e^{\frac{2\pi \i \alpha_{n}(r+1)}{\omega_{n}}}  \frac{\Gamma \left( \frac{r + 1}{\omega_{n}} \right)}{2M}\sum_{m=0}^{2M-1}
 \frac{g\left(t_m\right)}{w_m^r}\left(\frac{\left(t_m-t^{(n)}\right)^{\omega_n}}{f(t_m) -f_{n}}\right)^{(r + 1)/\omega_{n}},
\end{equation}
in which $t_m=t^{(n)}+w_m$ and $w_m=\rho\e^{\pi\i m/M}$. Typically this approximation converges exponentially fast with $M$. Note that in hyperasymptotics $n$ can be large and so we would need to take at least $M>n$.

For the proof of the results above we will obtain new integral representations for $T^{(n)} (z ;\alpha_{n})$ via several changes of variables.
The local behaviour (\ref{eq4a}) of $f(t)$ at the critical point $t^{\left(n \right)}$
suggests the parameterization
\begin{equation}\label{eq4}
s^{\omega_{n}} = z(f( t ) - f_{n} )
\end{equation}
of the integrand in \eqref{eq2} along $\mathscr{P}^{\left(n \right)} (\theta;\alpha_{n})$. Substitution of \eqref{eq4} in \eqref{eq2} yields
\begin{gather}\label{eq5}
\begin{split}
T^{(n)} (z ;\alpha_{n}) &= \omega_{n}z^{1/\omega_{n}} \int_0^{\infty \e^{\frac{2\pi \i \alpha_{n}}{\omega_{n}}} } \e^{ - s^{\omega_{n}} } g(t)\frac{\d t}{\d s}\d s \\
&=\omega_{n}\int_0^{\infty  \e^{\frac{2\pi \i \alpha_{n}}{\omega_{n}}} } \e^{ - s^{\omega_{n} }} \frac{\omega_{n}s^{\omega_{n}-1} }{z^{1-1/\omega_{n}}}\frac{g(t(s/z^{1/\omega_{n}} ))}{f'(t(s/z^{1/\omega_{n}} ))}\d s,
\end{split}
\end{gather}
where $t=t(s/z^{1/\omega_{n}} )$ is the unique solution of the equation \eqref{eq4} with $t(s/z^{1/\omega_{n}} ) \in \mathscr{P}^{\left(n \right)} (\theta;\alpha_{n})$. Since the contour $\mathscr{P}^{\left(n \right)} (\theta;\alpha_{n})$ does not pass through any of the saddle points of $f(t)$ other than $t^{\left(n\right)}$, the quantity
\begin{equation}\label{eq5a}
\frac{\omega_{n}s^{\omega_{n}-1} }{z^{1-1/\omega_{n}} }\frac{g(t(s/z^{1/\omega_{n}} ))}{f'(t(s/z^{1/\omega_{n}} ))} = \frac{\omega_{n}(f(t(s/z^{1/\omega_{n}} )) - f(t^{\left( n \right)} ))^{1-1/\omega_{n}} }{f'(t(s/z^{1/\omega_{n}} ))}g(t(s/z^{1/\omega_{n}} ))
\end{equation}
is an analytic function of $t$ in a neighbourhood of $\mathscr{P}^{(n)}( \theta;\alpha_{n})$. (We examine the analyticity of the factor $(f\left( t \right) - f_{n})^{1/\omega_{n}}$ in $\Delta^{\left(n\right)}$, after equation \eqref{eq7} below.) Whence, according to the residue theorem, the right-hand side of \eqref{eq5a} is\footnote{If $P(t)$ and $Q(t)$ are analytic in a neighbourhood of $t_0$ with $P(t_0) =0$ and $P'(t_0)\neq 0$, then $Q(t_0)/P'(t_0) = \mathop {\text{Res}}_{t = t_0} Q(t)/P(t)$.}  
\begin{equation*}
 \mathop {\text{Res}}\limits_{t = t(s/z^{1/\omega_{n}} )} \frac{{g( t)}}{{(f(t) - f_{n})^{1/\omega_{n}}  - s/z^{1/\omega_{n}} }} = \frac{1}{2\pi \i} \oint_{t(s/z^{1/\omega_{n}} )}  \frac{g(t)}{(f(t) - f_{n})^{1/\omega_{n}}  - s/z^{1/\omega_{n}}}\d t .
\end{equation*}
Substituting this expression into \eqref{eq5} leads to an alternative representation for the integral $T^{(n)}( z ;\alpha_{n})$ of the form
\begin{equation}\label{eq6}
T^{(n)} ( z;\alpha_{n} ) = \int_0^{\infty \e^{\frac{2\pi \i \alpha_{n}}{\omega_{n}}} }  \e^{ - s^{\omega_{n}} } \frac{\omega_{n}}{2\pi \i}\oint_{\Gamma ^{(n )}} \frac{g(t)}{(f(t) - f_n)^{1/\omega_{n}}  - s/z^{1/\omega_{n}}} \d t \d s.
\end{equation}
The infinite contour $\Gamma^{(n)}=\Gamma^{(n)}( \theta )$ encircles the path $\mathscr{P}^{(n)}( \theta ; \alpha_{n})$ in the positive direction within $\Delta^{(n)}$ (see Figure \ref{fig2}(a)). This integral will exist provided that $g(t)/f^{1/\omega_{n}}(t)$ decays sufficiently rapidly at infinity in $\Delta^{(n)}$. Otherwise, we can define $\Gamma^{(n)}( \theta )$ as a finite loop contour surrounding $t(s/z^{1/\omega_{n}} )$ and consider the limit
\begin{equation}\label{eq7}
\mathop {\lim }\limits_{S \to  \infty } \int_0^{ S \e^{\frac{2\pi \i \alpha_{n}}{\omega_{n}}} }  \e^{ - s^{\omega_{n}} } \frac{\omega_{n}}{2\pi \i}\oint_{\Gamma ^{(n)}} \frac{g(t)}{(f(t) - f_n)^{1/\omega_{n}}  - s/z^{1/\omega_{n}}} \d t \d s.
\end{equation}
The factor $(f(t) - f_{n})^{1/\omega_{n}}$ in \eqref{eq6} is carefully defined in the domain $\Delta^{(n)}$ as follows. First, we observe that $f(t) - f_{n}$ has an $\omega_{n}^{\rm th}$-order zero at $t=t^{(n)}$ and is non-zero elsewhere in $\Delta^{(n)}$ (because any point in $\Delta^{(n)}$, different from $t^{(n)}$, can be reached from $t^{\left(n\right)}$ by a path of descent). Second, $\mathscr{P}^{(n)}( \theta ;\alpha_{n})$ is a periodic function of $\theta$ with (least) period $2\pi \omega_{n}$. Hence, we may define the $\omega_{n}^{\rm th}$ root so that $(f(t) - f_{n})^{1/\omega_{n}}$ is a single-valued analytic function of $t$ in $\Delta^{\left(n\right)}$. The correct choice of the branch of $(f(t) - f_{n})^{1/\omega_{n}}$ is determined by the requirement that $\arg s=2\pi \alpha_{n}/\omega_{n}$ on $\mathscr{P}^{(n)} ( \theta ;\alpha_{n} )$, which can be fulfilled by setting $\arg \left[(f(t) - f_{n})^{1/\omega_{n}} \right] = (2\pi \alpha_{n}-\theta)/\omega_{n} $ for $t\in \mathscr{P}^{(n)}( \theta; \alpha_{n} )$. With any other definition of $(f(t) - f_{n})^{1/\omega_{n}}$, the representation \eqref{eq6} would be invalid.

Now, we employ the finite expression for non-negative integer $N$
\[
\frac{1}{1 - x} = \sum\limits_{r = 0}^{N - 1} {x^r }  + \frac{x^N}{1 - x},\qquad x \ne 1,
\]
to expand the denominator in \eqref{eq6} in powers of $s/[z(f( t ) - f_{n})]^{1/\omega_{n}}$. We thus obtain
\begin{align*}
T^{(n)} ( z ;\alpha_{n} ) = &\sum\limits_{r = 0}^{N - 1} \frac{1}{z^{ r/\omega_{n}}}\int_0^{\infty \e^{\frac{2\pi \i \alpha_{n}}{\omega_{n}}}}   \e^{ - s^{\omega_{n}} } s^r \frac{\omega_{n}}{{2\pi \i}}\oint_{\Gamma^{( n)} } \frac{g( t )}{(f(t) - f_{n} )^{ (r + 1)/\omega_{n}} }\d t \d s   \\
&+ R_N^{( n)} ( z ;\alpha_{n})
\end{align*}
with
\begin{gather}\label{eq8}
\begin{split}
R_N^{(n)} (z;\alpha_{n} ) =\; & \frac{\omega_{n}}{2\pi \i z^{N/\omega_{n} }}  \int_0^{\infty \e^{\frac{2\pi \i \alpha_{n}}{\omega_{n}}} } \e^{ - s^{\omega_{n}} } s^N \\
&\times\oint_{\Gamma ^{(n)}} \frac{g(t)}{(f(t) - f_{n})^{(N + 1)/\omega_{n}}}\frac{\d t}{1 - \frac{s}{\left(z(f(t) - f_{n})\right)^{1/\omega_{n}}}} \d s .
\end{split}
\end{gather}
Again, a limiting process is used in \eqref{eq8} if necessary. Throughout this work, if not stated otherwise, empty sums are taken to be zero.

For each term in the finite sum, the contour $\Gamma^{(n)}( \theta )$ can be shrunk into a small positively-oriented circle with centre $t^{(n)}$ and radius $\rho$, and we arrive at \eqref{FirstExpansion},
where the coefficients are given by \eqref{eq101} and \eqref{eq11}.

By expanding $\Gamma^{(n)}(\theta)$ to the boundary of $\Delta^{(n)}$ (see Section \ref{sec2}), we obtain
\begin{gather}\label{eq12}
\begin{split}
&R_N^{(n)} ( z; \alpha_{n}) =\frac{\omega_{n}}{2\pi \i z^{ N /\omega_{n}} }\sum_{m(n)} \int_0^{\infty \e^{\frac{2\pi \i\alpha_{n}}{\omega_{n}}}}  \e^{ - s^{\omega_{n} }} s^N \\
&\qquad\qquad \times\left(
\int_{\mathscr{P}^{(m)} (\theta _{nm}^ +  ,\alpha _{nm}^ +  )}  \frac{g(t)}{(f(t) - f_{n})^{ (N + 1)/\omega_{n}} }\frac{\d t}{1  -  \frac{s}{\left(z(f( t ) - f_{n})\right)^{ 1/\omega_{n}}} }\right.\\
 & \qquad\qquad\qquad\left.-\int_{\mathscr{P}^{(m)} (\theta _{nm}^ -  ,\alpha _{nm}^ -  )}  \frac{g(t)}{(f(t) - f_{n})^{ (N + 1)/\omega_{n}} }\frac{\d t}{1  -  \frac{s}{\left(z(f( t ) - f_{n})\right)^{ 1/\omega_{n}}} } \right)\d s.
\end{split}
\end{gather}


The expansion process is justified provided that (i) $f(t)$ and $g(t)$ are analytic in the domain $\Delta^{\left(n\right)}$, (ii) the quantity $g(t)/f^{(N + 1)/\omega_{n}}(t)$ decays sufficiently rapidly at infinity in $\Delta^{(n)}$, and (iii) there are no zeros of the denominator $1 - s/[z(f(t) - f_{n})]^{ 1/\omega_{n}}$ within the region $R$ through which the loop $\Gamma^{(n)}(\theta)$ is deformed. 

The first condition is already satisfied by prior assumption. The second condition is met by requiring that $g(t)/f^{(N + 1)/\omega_{n}}( t) = o(1/\left|t\right|)$ as $t\to \infty$ in $\Delta^{\left(n \right)}$ which we shall assume to be the case. The third condition is satisfied according to the following argument. The zeros of the denominator are those points of the $t$-plane for which $\arg\left[\e^{\i\theta } ( f( t) - f_n)\right]=2\pi \alpha_n$, in particular the points of the path $\mathscr{P}^{(n)} ( \theta  ; \alpha_n)$. Furthermore, no components of the set defined by the equation $\arg\left[\e^{\i\theta } ( f( t) - f_n)\right]=2\pi \alpha_n$ other than $\mathscr{P}^{(n)} ( \theta; \alpha_n)$ can lie within $\Delta^{(n)}$, otherwise $f(t)$ would have branch points along those components. By observing that $\mathscr{P}^{(n)} ( \theta; \alpha_n)$ is different for different values of $\theta \bmod 2\pi\omega_n$, we see that the locus of the zeros of the denominator $1 - s/[z(f(t) - f_n)]^{1/\omega_n}$ inside $\Delta^{(n)}$ is precisely the contour $\mathscr{P}^{(n)} ( \theta; \alpha_n)$, which is wholly contained within $\Gamma^{(n)}( \theta)$ and so these zeros are external to $R$.

We now consider the convergence of the double integrals in \eqref{eq12} further. To do this, we change variables from $t$ to $v$ by
\begin{equation}\label{eq13}
f( t ) - f_{n}= v \e^{(- \theta_{nm}^\pm + 2\pi \alpha _n) \i} ,
\end{equation}
where $v \geq \left|\mathcal{F}_{nm}^\pm\right|$. Since $\e^{ (\theta_{nm}^\pm - 2\pi \alpha _n) \i} (f( t) - f_{n})$ is a monotonic function of $t$ on the contour $\mathscr{P}^{(m)} (\theta _{nm}^\pm  ,\alpha _{nm}^\pm )$, corresponding to each value of $v$, there is a value of $t$, say $t_{\pm}\left(v\right)$, that satisfies \eqref{eq13}. The assumption \eqref{eq3} implies that the factor $\left[ 1 - s/[z(f(t) - f_n)]^{1/\omega_{n}} \right]^{ - 1}$ in \eqref{eq12} is bounded above by a constant. Hence, the convergence of the double integrals in \eqref{eq12} will be assured provided the real double integrals
\[
\int_0^{\infty } \int_{\left| {\mathcal{F}_{nm}^\pm } \right|}^{\infty } \frac{ \e^{ - |s|^{\omega_{n}}} |s|^N}{v^{(N + 1)/\omega_{n}}}\left| \frac{g( t_ \pm  (v))}{f'(t_\pm (v))} \right|\d v \d|s|
\]
exist. In turn, these real double integrals will exist if and only if the single integrals  
\begin{equation}\label{eq14}
\int_{\left| {\mathcal{F}_{nm}^\pm } \right|}^{\infty } \frac{1}{v^{(N + 1)/\omega_{n}}}\left| \frac{g(t_\pm(v))}{f'(t_\pm (v))} \right|\d v 
\end{equation}
exist. Henceforth, we assume that the integrals in \eqref{eq14} exist for each of the adjacent contours.

On each of the contours $\mathscr{P}^{(m)} (\theta _{nm}^ \pm  ,\alpha _{nm}^ \pm  )$ in \eqref{eq12}, we perform the change of variable from $s$ and $t$ to $u$ and $t$ via
\[
s^{\omega_{n}}  = u(f( t ) - f_{n}) = \mathcal{F}_{nm}^\pm u + u (f( t ) - f_{m})
\]
to obtain
\begin{gather}\label{eq14c}
\begin{split}
&R_N^{(n)} ( z;\alpha_{n} ) =  \sum_{m(n)} \frac{z^{(1-N)/\omega_n}}{2\pi \i} \\
&\quad\times\left(  
\int_0^{\infty \e^{\i \theta _{nm}^+ } } \frac{ \e^{ - \mathcal{F}_{nm}^+ u} u^{\frac{{N + 1}}{{\omega _n }} - 1}}{z^{1/\omega_n} - u^{1/\omega_n}}
\int_{\mathscr{P}^{(m)} (\theta _{nm}^+  ,\alpha _{nm}^+  ) } \e^{ - u(f(t) - f_m )} g(t) \d t \d u\right.\\
&\qquad\quad- \left.\int_0^{\infty \e^{\i \theta _{nm}^-} } \frac{ \e^{ - \mathcal{F}_{nm}^- u} u^{\frac{{N + 1}}{{\omega _n }} - 1}}{z^{1/\omega_n} - u^{1/\omega_n}}
\int_{\mathscr{P}^{(m)} (\theta _{nm}^-  ,\alpha _{nm}^-  ) } \e^{ - u(f(t) - f_m )} g(t)\d t \d u\right).
\end{split}
\end{gather}
This change of variable is permitted because the infinite double integrals in \eqref{eq12} are assumed to be absolutely convergent, which is 
a consequence of the requirement that the integrals \eqref{eq14} exist. Hence the exact remainder of the expansion \eqref{FirstExpansion}
about the critical point $t^{(n)}$ is expressible in terms of similar integrals over infinite contours emanating from the adjacent saddles $t^{(m)}
$ as
\begin{gather}\label{eq14b}
\begin{split}
R_N^{(n)} ( z;\alpha_{n} ) =  &\sum_{m(n)} \frac{z^{(1-N)/\omega_n}}{2\pi \i\omega_m} 
\left(\int_0^{\infty \e^{\i\theta _{nm}^+} } \frac{ \e^{ - \mathcal{F}_{nm}^+ u} u^{\frac{N+1}{{\omega _n }} -\frac{1}{{\omega _m }}  - 1}}{z^{1/\omega_n} - u^{1/\omega_n}}T^{(m)} (u;\alpha _{nm}^+ )\d u\right.\\
&\qquad\qquad\quad -\left.\int_0^{\infty \e^{\i\theta _{nm}^-} } \frac{\e^{ - \mathcal{F}_{nm}^- u} u^{\frac{N+1}{{\omega _n }} -\frac{1}{{\omega _m }}  - 1}}{z^{1/\omega_n} - u^{1/\omega_n}}T^{(m)} (u;\alpha _{nm}^- )\d u\right).
\end{split}
\end{gather}
Since $\theta _{nm}^ + = \theta _{nm}^ - +2\pi\omega_n$, a simple change of integration variable in \eqref{eq14c} then yields
\begin{gather}\label{eq14d}
\begin{split}
&R_N^{(n)} ( z;\alpha_{n} ) = \sum_{m(n)} \frac{z^{(1-N)/\omega_n}}{2\pi \i}  \\
&\qquad\times\left(\int_0^{\infty \e^{\i \theta _{nm}^+ } } \frac{ \e^{ - \mathcal{F}_{nm}^+ u} u^{\frac{N + 1}{\omega _n } - 1}}{z^{1/\omega_n} - u^{1/\omega_n}}
\int_{\mathscr{P}^{(m)} (\theta _{nm}^+  ,\alpha _{nm}^+  ) } \e^{ - u(f(t) - f_m )} g(t) \d t \d u\right.\\
&\qquad\qquad\left. - \int_0^{\infty \e^{\i \theta _{nm}^+} } \frac{ \e^{ - \mathcal{F}_{nm}^+ u} u^{\frac{{N + 1}}{{\omega _n }} - 1}}{z^{1/\omega_n} - u^{1/\omega_n}}
\int_{\mathscr{P}^{(m)} (\theta _{nm}^+  ,\beta_{nm} ) } \e^{ - u(f(t) - f_m )} g(t)\d t \d u \right).
\end{split}
\end{gather}
The path $\mathscr{P}^{(m)} (\theta _{nm}^+  ,\beta_{nm} )$ is geometrically identical to $\mathscr{P}^{(m)} (\theta _{nm}^-  ,\alpha _{nm}^-  )$, and since the angle of the slope of $\mathscr{P}^{(m)} (\theta _{nm}^-  ,\alpha _{nm}^-  )$ to the positive real axis at $t^{(m)}$ is $2\pi/\omega_m$ higher than the corresponding angle of $\mathscr{P}^{(m)} (\theta _{nm}^+  ,\alpha _{nm}^+  )$, we find (cf. \eqref{alphanm})
\begin{align*}
\beta _{nm}  &= \frac{{\theta _{nm}^ +   + \arg (f^{(\omega _m )} (t^{(m)} )) + \omega _m (\varphi ^ +   + 2\pi /\omega _m )}}{{2\pi }} \\
&= \frac{{\theta _{nm}^ +   + \arg (f^{(\omega _m )} (t^{(m)} )) + \omega _m \varphi ^ +  }}{{2\pi }} + 1 ~~= \alpha _{nm}^ +   + 1.
\end{align*}

With the notation in \eqref{double}, integral representation \eqref{eq14d} can be written as \eqref{Step1}.
The observation that
\begin{equation}\label{Step2}
R_{N}^{(n)}(z;\alpha_n+1)=\sum_{m(n)}\frac{z^{(1-N)/\omega_n}}{2\pi\i\omega_m}\!\int_0^{\infty\e^{\i(\theta_{nm}^{+}+2\pi)}}
\!\!\frac{\e^{-\mathcal{F}_{nm}^+ u} u^{\frac{N+1}{\omega_n}-\frac{1}{\omega_m}-1}}{z^{1/\omega_n}-u^{1/\omega_n}}
{\bf T}^{(m)}(u;\alpha_{nm}^{+}+1)\d u,
\end{equation}
will also be useful.

In previous publications  \cite{BH91,Howls97} there were issues with the exact sign of the terms on the right-hand side of \eqref{Step1}.  
These were referred to as ``orientation anomalies''.  Here we do not encounter these issues because of the careful definitions of the 
phases on the contours (\ref{alpha}), (\ref{alphanm}).

\section{Hyperasymptotic iteration of the exact remainder}\label{sec4}

In this section we re-expand the exact remainder terms (\ref{Step1}) and (\ref{Step2}) to derive a template for hyperasymptotic calculations.

First, we begin by defining a set of universal, but generalised, hyperterminant functions ${\bf F}^{(j)}$, that form the basis of the template.

Let us introduce the notation $\displaystyle \int_0^{[\eta]}=\int_0^{\infty \e^{\i\eta}}$. Then, for $k$ a non-negative integer, we define
\begin{align}
& \HypTermZero{z} := 1,\qquad\qquad
\HypTermOne{z}{M_0}{\omega_0}{\sigma _0} 
:= \int_0^{\left[ {\pi  - \arg \sigma _0 } \right]} \frac{\e^{\sigma _0 t_0 } t_0^{M_0  - 1} }{z^{1/\omega _0 }  - t_0^{1/\omega _0 } }\d t_0,\nonumber\\
& \HypTermK{k+1}{z}{%
\begin{array}{c}
   {M_0 ,}  \\
   {\omega _0 ,}  \\
   {\sigma _0  ,}  \\
\end{array}\begin{array}{c}
   { \ldots ,}  \\
   { \ldots ,}  \\
   { \ldots ,}  \\
\end{array}\begin{array}{c}
   {M_k }  \\
   {\omega _k }  \\
   {\sigma _k }  \\
\end{array}}\label{defhyperterminant}\\
&:= \int_0^{\left[ {\pi  - \arg \sigma _0 } \right]}\!\!\!\!\!\!\cdots \int_0^{\left[ {\pi  - \arg \sigma _k } \right]} \frac{\e^{\sigma _0 t_0  +  \cdots  + \sigma _k t_k } t_0^{M_0  - 1}  \cdots t_k^{M_k  - 1} }{(z^{1/\omega_0 }  - t_0^{1/\omega_0 } )(t_0^{1/\omega_1 }  - t_1^{1/\omega_1 } ) \cdots (t_{k - 1}^{1/\omega_{k} }  - t_k^{1/\omega_k } )}\d t_k  \cdots \d t_0,\nonumber
\end{align}
for arbitrary sets of complex numbers $M_0,\ldots,M_k$ and $\sigma_0,\ldots,\sigma_k$ such that $\mathop{\rm Re}(M_j)>1/\omega_j$ and $\sigma_j \neq 0$ for $j=0,\ldots,k$, and for an arbitrary set of positive integers $\omega_0,\ldots,\omega_k$. The multiple integrals converge when $|\arg(\sigma_0 z)|<\pi\omega_0$. The ${\bf F}^{(j)}$ is termed a ``generalised $j^{\rm th}$-level hyperterminant''. If $\omega_0=\cdots=\omega_{j-1}=1$, ${\bf F}^{(j)}$ reduces to the much simpler $j^{\rm th}$-level hyperterminant $F^{(j)}$ discussed in the paper \cite{OD98c}.

Note that in the case that two successive $\sigma$'s have the same phase the choice of integration path over the poles in \eqref{defhyperterminant} needs to be defined more carefully. In those cases we can define the hyperterminant via a limit. For example
\begin{equation}\label{hyptermlimit}
\lim_{\varepsilon\to0^+}\HypTermK{k+1}{z}{%
\begin{array}{c}
   {M_0 ,}  \\
   {\omega _0 ,}  \\
   {\sigma _0 \e^{ - k\varepsilon \i} ,}  \\
\end{array}\begin{array}{c}
   {M_1 ,}  \\
   {\omega _1 ,}  \\
   {\sigma _1 \e^{ - (k - 1)\varepsilon \i} ,}  \\
\end{array}\begin{array}{c}
   { \ldots ,}  \\
   { \ldots ,}  \\
   { \ldots ,}  \\
\end{array}\begin{array}{c}
   {M_{k - 1} ,}  \\
   {\omega _{k - 1} ,}  \\
   {\sigma _{k - 1} \e^{ - \varepsilon \i} ,}  \\
\end{array}\begin{array}{c}
   {M_k }  \\
   {\omega _k }  \\
   {\sigma _k }  \\
\end{array}}
\end{equation}
is an option. Other limits are also possible.

The efficient computation of these generalised hyperterminant functions is outlined in Appendix \ref{appA}.

\subsection{Superasymptotics and optimal number of terms}\label{optimalSect}

A necessary step in hyperasymptotic re-expansions is to determine the ``optimal'' number of terms in the original Poincar\'e expansion \eqref{FirstExpansion}, defined as the index of the least term in magnitude.

For this section it reasonable to denote the original number of terms in the truncated asymptotic expansion as $N=N_0^{(n)}$ and we denote the associated remainder as $R_{0}^{(n)}(z;\alpha_n)$.
With this notation the integrands in \eqref{Step1} will have a factor $u^{N_0^{(n)}/\omega_n}$. Therefore, 
when $N_0^{(n)}$ is large, the main contribution to the integrals in \eqref{Step1} comes from infinity where ${\bf T}^{(m)}(u;\alpha_{nm}^{+})=\mathcal{O}(1)$. In the case that $z$ and $u$ are collinear, i.e., 
on a Stokes line, it is well known (see, e.g.,  \cite[\S8]{Boyd93} or \cite[\S5]{OO94}) that the Stokes phenomenon produces 
an extra factor of $\mathcal{O}\left( \sqrt {N_0^{(n)}}\right)$ when estimating 
$R_{0}^{(n)}(z;\alpha_n)$ (see also the proof of Proposition \ref{hyperbound}). Thus, we have
\[
R_0^{(n)} (z;\alpha _n ) = \sqrt {N_0^{(n)}} \frac{\Gamma \left( \frac{N_0^{(n)} + 1}{\omega _n } \right)}{\left| z \right|^{\frac{N_0^{(n)}}{\omega _n }} }
\sum\limits_{m(n)} \frac{1}{|\mathcal{F}_{nm}^ +|^{\frac{N_0^{(n)}}{\omega _n }} \left(N_0^{(n)}\right)^{\frac{1}{\omega _m }} } \mathcal{O}(1),
\]
for large $N_0^{(n)}$ and $\theta_{nm_1}^ - \leq \theta \leq \theta_{nm_2}^ +$. Let $N_0^{(n)} = \eta_0^{(n)} \omega _n \left| z \right|+\nu_0^{(n)}$ with $\nu_0^{(n)}$ being bounded. Then, with the help of Stirling's formula,
\begin{equation}\label{superasymp}
R_0^{(n)} (z;\alpha _n ) = \e^{ - \eta _0^{(n)} |z|} \sum\limits_{m(n)} \left|z\right|^{\frac{1}{\omega _n}-\frac{1}{\omega _m}} \left( \frac{\eta _0^{(n)}}{\left|\mathcal{F}_{nm}^+\right|} \right)^{\eta _0^{(n)} |z|} \mathcal{O}(1),
\end{equation}
as $|z|\to\infty$ in the sector $\theta_{nm_1}^ - \leq \theta \leq \theta_{nm_2}^ +$. For a fixed $m$ the magnitude of the right-hand side of \eqref{superasymp} is minimal in the case that $\eta_0^{(n)}=\left|\mathcal{F}_{nm}^+\right|$.
Since we sum over all the adjacent saddles we obtain that for the optimal number of terms we have $\eta_0^{(n)}=r_0^{(n)}:=\min_{m(n)}\left|\mathcal{F}_{nm}^+\right|$, and with that choice we have
\begin{equation}\label{superasymp2}
R_0^{(n)} (z;\alpha _n ) =\e^{ -r_0^{(n)} |z|} \left|z\right|^{\frac{1}{\omega _n}-\frac{1}{\widetilde\omega}} \mathcal{O}(1),
\end{equation}
as $|z|\to\infty$ in the sector $\theta_{nm_1}^ - \leq \theta \leq \theta_{nm_2}^ +$ with $\widetilde\omega=\max_j\omega_j$. 

In the hyperasymptotic process below, we will re-expand this remainder and each of these re-expansions will be truncated and re-expanded 
and so on.  Correspondingly we have to determine the number of terms to take in the original expansion $N_0^{(n)}$, in the first re-expansions
$N_1^{(m)}$, and so on. The criterion for determining the ``optimal'' $N_0^{(n)}$, $N_1^{(m)}$, \dots, is that the overall error 
obtained by summing all the contributing expansions should be minimised. This may be determined from considering estimates such as 
\eqref{superasymp} and \eqref{Step4a}, \eqref{Step5a} below. The procedure for determining these optimal numbers of terms is very similar 
to that of \cite{OD98b}, and may be summarised as follows.

Let $G=(V,E)$ be a graph with for the vertices $V$ all the $f_j$ and for the edges $E=\left\{ (f_m,f_n) : t^{(m)} {\rm ~is~adjacent~to~} t^{(n)}\right\}$.
We define $r^{(n)}_k$ to be the length of the shortest path of $k+1$ steps in this graph starting at $t^{(n)}$. For a hyperasymptotic expansion of Level $k$ the optimal number of terms is
\begin{equation}\label{optimal}
N_0^{(m_0)}=\eta_0^{(m_0)}\omega_{m_0}|z|+\nu_0^{(m_0)},\qquad \ldots,\qquad
N_k^{(m_k)}=\eta_k^{(m_k)}\omega_{m_k}|z|+\nu_k^{(m_k)},
\end{equation}
with $m_0=n$, in which
$$\eta_0^{(m_0)} :=r^{(m_0)}_{k},\qquad \eta_j^{(m_j)}:=\max\left(0,\eta_{j-1}^{(m_{j-1})}-|\mathcal{F}_{m_{j-1}m_j}|\right),\qquad j=1,\ldots,k,$$
and the $\nu_j$ are all bounded as $|z|\to\infty$, with estimate 
\begin{equation}\label{superasymp3}
R_k^{(n)} (z;\alpha _n ) =\e^{ -r_k^{(n)} |z|} \left|z\right|^{\frac{1}{\omega _n}-\frac{1}{\widetilde\omega}} \mathcal{O}(1),
\end{equation}
for the remainder
as $|z|\to\infty$ in the sector $\theta_{nm_1}^ - \leq \theta \leq \theta_{nm_2}^ +$. The main difference from the results in \cite{OD98b} is that 
here in \eqref{optimal} we have the extra factors $\omega_j$.

\subsection{Level 1 hyperasymptotics}

We now derive the Level 1 hyperasymptotic expansion. In the integral representation \eqref{Step1} for this remainder we substitute 
\eqref{FirstExpansion} into the ${\bf T}^{(m)}$ function. We obtain the re-expansion
\begin{gather}\begin{split}\label{Step3}
R_{0}^{(n)}(z;\alpha_n)=&
\sum_{m(n)}\frac{z^{(1-N_0^{(n)})/\omega_n}}{2\pi\i\omega_m} \sum_{r=0}^{N_1^{(m)}-1}{\bf T}_r^{(m)}(\alpha_{nm}^{+})
\HypTermOne{z}{\frac{N_0^{(n)}+1}{\omega_n}-\frac{r+1}{\omega_m}}{\omega_n}{|\mathcal{F}_{nm}^+|\e^{\i(\pi-\theta_{nm}^{+})}}\\\
&+R_{1}^{(n)}(z;\alpha_n).
\end{split}\end{gather}
The remainder $R_{1}^{(n)}(z;\alpha_n)$ depends on the number of terms $N_0^{(n)}$ and $N_1^{(m)}$ and can be represented as
\begin{gather}\begin{split}\label{Step4}
&R_{1}^{(n)}(z;\alpha_n)=\sum_{m(n)}\sum_{\ell(m)}\frac{z^{(1-N_0^{(n)})/\omega_n}}{\left(2\pi\i\right)^2\omega_m\omega_\ell} \\
&\quad\times\Bigg(
\int_0^{\infty\e^{\i\theta_{nm}^{+}}}\int_0^{\infty\e^{\i\theta_{nm\ell}^{+}}}
\frac{\e^{-\mathcal{F}_{nm}^+ u-\mathcal{F}_{m\ell}^+ v} u^{\frac{N_0^{(n)}+1}{\omega_n}-\frac{N_1^{(m)}}{\omega_m}-1}
v^{\frac{N_1^{(m)}+1}{\omega_m}-\frac{1}{\omega_\ell}-1}}{\left(z^{1/\omega_n}-u^{1/\omega_n}\right)\left(u^{1/\omega_m}-v^{1/\omega_m}\right)}\\
&\qquad\qquad\qquad\qquad\qquad\qquad\times{\bf T}^{(\ell)}(v;\alpha_{nm\ell}^{+})\d v\d u\\
&\quad\qquad-\int_0^{\infty\e^{\i\theta_{nm}^{+}}}\int_0^{\infty\e^{\i\theta_{nm\ell}^{+}+2\pi\i}}
\frac{\e^{-\mathcal{F}_{nm}^+ u-\mathcal{F}_{m\ell}^+ v} u^{\frac{N_0^{(n)}+1}{\omega_n}-\frac{N_1^{(m)}}{\omega_m}-1}
v^{\frac{N_1^{(m)}+1}{\omega_m}-\frac{1}{\omega_\ell}-1}}{\left(z^{1/\omega_n}-u^{1/\omega_n}\right)\left(u^{1/\omega_m}-v^{1/\omega_m}\right)}\\
&\qquad\qquad\qquad\qquad\qquad\qquad\times{\bf T}^{(\ell)}(v;\alpha_{nm\ell}^{+}+1)\d v\d u\Bigg),
\end{split}\end{gather}
in which $\theta_{nm\ell}^{+}(\theta_{nm}^{+})$ corresponds to the path $\mathscr{P}^{\left(n\right)}( \theta_{nm}^{+} ;\alpha_{nm}^{+})$ and is defined similarly as $\theta_{nm}^{+}=\theta_{nm}^{+}(\theta)$. The $\alpha_{nm\ell}^{+}$ is the corresponding $\alpha_{nm}^{+}$, which is defined \eqref{alphanm}. In this derivation we have used the observation \eqref{Step2}.

We can estimate the remainder $R_1^{(n)} (z;\alpha _n ) $ in a similar way as we did $R_0^{(n)} (z;\alpha _n ) $, and one finds
\begin{align*}
R_1^{(n)} (z;\alpha _n ) = &\frac{1}{{|z|^{\frac{{N_0^{(n)} }}{{\omega _n }}} }}  \sum_{m(n)} \sqrt {\left( {N_0^{(n)}  - N_1^{(m)} } \right)N_1^{(m)} } \\ 
& \times \frac{{\Gamma \left( {\frac{{N_0^{(n)}  + 1}}{{\omega _n }} - \frac{{N_1^{(m)}  + 1}}{{\omega _m }}} \right)\Gamma \left( {\frac{{N_1^{(m)}  + 1}}{{\omega _m }}} \right)}}{%
{|\mathcal{F}_{nm}^ +  |^{\frac{{N_0^{(n)} }}{{\omega _n }} - \frac{{N_1^{(m)} }}{{\omega _m }}} }} 
\sum_{\ell (m)} {\frac{1}{{|\mathcal{F}_{m\ell }^ +  |^{\frac{{N_1^{(m)} }}{{\omega _m }}} \left( {N_1^{(m)} } \right)^{\frac{1}{{\omega _\ell  }}} }}}  \mathcal{O}(1) .
\end{align*}
Then
\begin{gather}\begin{split}\label{Step4a}
R_1^{(n)} (z;\alpha _n ) =  \e^{ - \eta _0^{(n)} |z|} &\sum_{m(n)} \left( {\frac{{\eta _0^{(n)}  - \eta _1^{(m)} }}{{|\mathcal{F}_{nm}^ +|}}} \right)^{(\eta _0^{(n)}  - \eta _1^{(m)} )|z|} \\
&\times \sum_{\ell (m)} \left|z\right|^{\frac{1}{\omega _n}-\frac{1}{\omega _\ell}} \left( {\frac{{\eta _1^{(m)} }}{{|\mathcal{F}_{m\ell }^ +|}}} \right)^{\eta _1^{(m)} \left| z \right|} \mathcal{O}(1),
\end{split}\end{gather}
as $|z|\to \infty$ in the sector $\theta_{nm_1}^ - \leq \theta \leq \theta_{nm_2}^ +$. For fixed $m$ and $\ell$, using a similar approach to 
Subsection \ref{optimalSect} above, it is easy to show that the optimal number of terms is obtained when 
$\eta _0^{(n)}  - \eta _1^{(m)}=\left|\mathcal{F}_{nm}^ +\right|$ and $\eta _1^{(m)}=\left|\mathcal{F}_{m\ell }^ +\right|$. 

Rigorous bounds for Level 1 hyperterminants are derived in Appendix \ref{appB}.

\subsection{Level 2 hyperasymptotics}

The Level 2 hyperasymptotic expansion is now derived by re-expanding the Level 1 expansion.  Again we substitute \eqref{FirstExpansion} 
into the ${\bf T}^{(\ell)}$ 
functions on the right-hand side of \eqref{Step4} and obtain the re-expansion
\begin{gather}\begin{split}\label{Step5}
&R_{1}^{(n)}(z;\alpha_n)=\sum_{m(n)}\sum_{\ell(m)}\frac{z^{(1-N_0^{(n)})/\omega_n}}{\left(2\pi\i\right)^2\omega_m\omega_\ell}\sum_{r=0}^{N_2^{(\ell)}-1}\\
&\qquad \left\{{\bf T}_r^{(\ell)}(\alpha_{nm\ell}^{+})
\HypTermTwo{z}{\frac{N_0^{(n)}+1}{\omega_n}-\frac{N_1^{(m)}}{\omega_m}}{\frac{N_1^{(m)}+1}{\omega_m}-\frac{r+1}{\omega_\ell}}{\omega_n}{\omega_m}
{|\mathcal{F}_{nm}^+|\e^{\i(\pi-\theta_{nm}^{+})}}{|\mathcal{F}_{m\ell}^+|\e^{\i(\pi-\theta_{nm\ell}^{+})}}\right.\\
&\qquad\quad \left.-{\bf T}_r^{(\ell)}(\alpha_{nm\ell}^{+}+1)
\HypTermTwo{z}{\frac{N_0^{(n)}+1}{\omega_n}-\frac{N_1^{(m)}}{\omega_m}}{\frac{N_1^{(m)}+1}{\omega_m}-\frac{r+1}{\omega_\ell}}{\omega_n}{\omega_m}
{|\mathcal{F}_{nm}^+|\e^{\i(\pi-\theta_{nm}^{+})}}{|\mathcal{F}_{m\ell}^+|\e^{\i(-\pi-\theta_{nm\ell}^{+})}}\right\}\\
&\qquad\quad+R_{2}^{(n)}(z;\alpha_n).
\end{split}\end{gather}
We also obtain an exact integral representation for the remainder, and this can be used to obtain the estimate
\begin{gather}\begin{split}\label{Step5a}
R_2^{(n)} (z;\alpha _n ) = \e^{ - \eta _0^{(n)} |z|}  &\sum_{m(n)} \left( {\frac{{\eta _0^{(n)}  - \eta _1^{(m)} }}{{|\mathcal{F}_{nm}^ +|}}} \right)^{(\eta _0^{(n)}  - \eta _1^{(m)} )|z|} 
\!\!\!\sum_{\ell(m)} \left( {\frac{{\eta _1^{(m)}  - \eta _2^{(\ell)} }}{{|\mathcal{F}_{m\ell}^ +|}}} \right)^{(\eta _1^{(m)}  - \eta _2^{(\ell)} )|z|}\\
 &\qquad \times  \sum_{k (\ell)} \left|z\right|^{ \frac{1}{{\omega _n }}- \frac{1}{{\omega _k  }}} \left( {\frac{{\eta _2^{(\ell)} }}{{|\mathcal{F}_{\ell k}^ +|}}} \right)^{\eta _2^{(\ell)} \left| z \right|}   \mathcal{O}(1),
\end{split}\end{gather}
as $|z|\to \infty$ in the sector $\theta_{nm_1}^ - \leq \theta \leq \theta_{nm_2}^ +$.

\subsection{Level 3 hyperasymptotics}

We can continue with this process and will obtain at Level 3 the expansion
\begin{gather}\begin{split}\label{Step6}
&R_{2}^{(n)}(z;\alpha_n)=\sum_{m(n)}\sum_{\ell(m)}\sum_{k(\ell)}\frac{z^{(1-N_0^{(n)})/\omega_n}}{\left(2\pi\i\right)^3\omega_m\omega_\ell\omega_k}\sum_{r=0}^{N_3^{(k)}-1}\\
& \left({\bf T}_r^{(k)}(\alpha_{nm\ell k}^{+})
\HypTermK{3}{z}{\begin{array}{c}{\frac{N_0^{(n)}+1}{\omega_n}-\frac{N_1^{(m)}}{\omega_m},}\\ {\omega_n,}\\ {|\mathcal{F}_{nm}^+|\e^{\i(\pi-\theta_{nm}^{+})},}\\ \end{array}
\begin{array}{c}{\frac{N_1^{(m)}+1}{\omega_m}-\frac{N_2^{(\ell)}}{\omega_\ell},}\\ {\omega_m,}\\ {|\mathcal{F}_{m\ell}^+|\e^{\i(\pi-\theta_{nm\ell}^{+})},}\\ \end{array}
\begin{array}{c}{\frac{N_2^{(\ell)}+1}{\omega_\ell}-\frac{r+1}{\omega_k}}\\ {\omega_\ell}\\ {|\mathcal{F}_{\ell k}^+|\e^{\i(\pi-\theta_{nm\ell k}^{+})}}\\ \end{array}}
\right.\\
& -{\bf T}_r^{(k)}(\alpha_{nm\ell k}^{+}+1)
\HypTermK{3}{z}{\begin{array}{c}{\frac{N_0^{(n)}+1}{\omega_n}-\frac{N_1^{(m)}}{\omega_m},}\\ {\omega_n,}\\ {|\mathcal{F}_{nm}^+|\e^{\i(\pi-\theta_{nm}^{+})},}\\ \end{array}
\begin{array}{c}{\frac{N_1^{(m)}+1}{\omega_m}-\frac{N_2^{(\ell)}}{\omega_\ell},}\\ {\omega_m,}\\ {|\mathcal{F}_{m\ell}^+|\e^{\i(\pi-\theta_{nm\ell}^{+})},}\\ \end{array}
\begin{array}{c}{\frac{N_2^{(\ell)}+1}{\omega_\ell}-\frac{r+1}{\omega_k}}\\ {\omega_\ell}\\ {|\mathcal{F}_{\ell k}^+|\e^{\i(-\pi-\theta_{nm\ell k}^{+})}}\\ \end{array}}
\\
& -{\bf T}_r^{(k)}(\alpha_{nm\ell k}^{+}+1)
\HypTermK{3}{z}{\begin{array}{c}{\frac{N_0^{(n)}+1}{\omega_n}-\frac{N_1^{(m)}}{\omega_m},}\\ {\omega_n,}\\ {|\mathcal{F}_{nm}^+|\e^{\i(\pi-\theta_{nm}^{+})},}\\ \end{array}
\begin{array}{c}{\frac{N_1^{(m)}+1}{\omega_m}-\frac{N_2^{(\ell)}}{\omega_\ell},}\\ {\omega_m,}\\ {|\mathcal{F}_{m\ell}^+|\e^{\i(-\pi-\theta_{nm\ell}^{+})},}\\ \end{array}
\begin{array}{c}{\frac{N_2^{(\ell)}+1}{\omega_\ell}-\frac{r+1}{\omega_k}}\\ {\omega_\ell}\\ {|\mathcal{F}_{\ell k}^+|\e^{\i(-\pi-\theta_{nm\ell k}^{+})}}\\ \end{array}}
\\
&\left.+{\bf T}_r^{(k)}(\alpha_{nm\ell k}^{+}+2)
\HypTermK{3}{z}{\begin{array}{c}{\frac{N_0^{(n)}+1}{\omega_n}-\frac{N_1^{(m)}}{\omega_m},}\\ {\omega_n,}\\ {|\mathcal{F}_{nm}^+|\e^{\i(\pi-\theta_{nm}^{+})},}\\ \end{array}
\begin{array}{c}{\frac{N_1^{(m)}+1}{\omega_m}-\frac{N_2^{(\ell)}}{\omega_\ell},}\\ {\omega_m,}\\ {|\mathcal{F}_{m\ell}^+|\e^{\i(-\pi-\theta_{nm\ell}^{+})},}\\ \end{array}
\begin{array}{c}{\frac{N_2^{(\ell)}+1}{\omega_\ell}-\frac{r+1}{\omega_k}}\\ {\omega_\ell}\\ {|\mathcal{F}_{\ell k}^+|\e^{\i(-3\pi-\theta_{nm\ell k}^{+})}}\\ \end{array}}\right)\\
&+R_{3}^{(n)}(z;\alpha_n).
\end{split}\end{gather}

An estimate for the remainder $R_{3}^{(n)}(z;\alpha_n)$, similar to those of \eqref{superasymp}, \eqref{Step4a} and \eqref{Step5a} may be 
obtained, and further iterations to higher hyper-levels derived.  We spare the reader these details as the pattern should now be clear.  

Initially, this expansion might seem over complicated.  However inspection of the terms shows that once we have line two of \eqref{Step6} 
the details of the other lines can be easily deduced. It follows from \eqref{double} and \eqref{eq11}
that the coefficients follow from the coefficients in line 2 by just multiplying by a simple exponential. The generalised hyperterminants only 
differ by a change in the phases of two (bottom centre and right) arguments.

\subsection{Late coefficients and resurgence} The re-expansion \eqref{Step3} is suitable for obtaining an asymptotic expansion for the late (large-$N$) coefficients $T_N^{(n)}(\alpha_n)$. Indeed, if we combine the identity 
$$T_N^{(n)}(\alpha_n)=z^{N/\omega_n}\left(R_N^{(n)}(z;\alpha_n)-R_{N+1}^{(n)}(z;\alpha_n)\right)$$
with \eqref{Step3}, we deduce
\begin{gather}\begin{split}\label{LateTerms}
T_N^{(n)}(\alpha_n)=&
\sum_{m(n)}\frac{1}{2\pi\i\omega_m} \sum_{r=0}^{N_1^{(m)}-1}{\bf T}_r^{(m)}(\alpha_{nm}^{+})
\frac{\e^{\i\theta_{nm}^+\left(\frac{N+1}{\omega_n}-\frac{r+1}{\omega_m}\right)}\Gamma\left(\frac{N+1}{\omega_n}-\frac{r+1}{\omega_m}\right)}
{|\mathcal{F}_{nm}^+|^{\frac{N+1}{\omega_n}-\frac{r+1}{\omega_m}}}\\
&+\widetilde{R}_1^{(n)}(N;\alpha_n).
\end{split}\end{gather}
Note that the coefficients in this expansion are the coefficients of the asymptotic expansions of integrals over doubly infinite contours passing 
through the adjacent saddles, a manifestation of ``resurgence".   The form (\ref{LateTerms}) is of a generalised sum of factorials over 
powers.  Note the careful representation of the phases of the singulants.  Various special cases of \eqref{LateTerms} were derived, using non-rigorous methods, by Dingle (see \cite[Ch. VII]{Dingle73}, 
including exercises). See also \cite{BH91}, \cite{Howls92}.

When we eliminate $|z|$ in the definitions \eqref{optimal} we obtain for the optimal numbers of terms in \eqref{LateTerms} that
$$N_1^{(m)}=\frac{\eta_1^{(m)}\omega_m}{\eta_0^{(n)}\omega_n}N+ \mathcal{O}(1),$$
as $N\to\infty$.

In the swallowtail example below we shall illustrate how this result can be used to determine the adjacency of the saddles algebraically 
rather than geometrically.

\section{Error bounds}\label{sec5}

In this section we derive rigorous, novel and sharp error bounds for the exact remainder $R_N^{(n)} (z;\alpha _n )$ of asymptotic expansions of the form \eqref{FirstExpansion} derived from integrals of the class \eqref{eq1}.

The remainder term \eqref{eq12} can be written as
\begin{gather}\begin{split}
&R_N^{(n)} (z;\alpha _n ) \\ 
&\quad= \frac{{\omega _n }}{{2\pi \i z^{N/ \omega _n} }} \sum_{m(n)} 
\int_{\mathscr{C}^{(m)} (\theta _{nm}^ +)} {\frac{g(t)}{{(f(t) - f_n )^{(N + 1)/\omega _n} }}
\int_0^{\infty \e^{\frac{{2\pi \i\alpha _n }}{{\omega _n }}} }\!\!\! {\frac{{ \e^{ - s^{\omega _n } }s^N }}{{1 - \frac{s}{\left( {z(f(t) - f_n )} \right)^{1/\omega _n}} }}\d s} \d t} \\
&\quad = \frac{{\e^{2\pi \i\frac{{N + 1}}{{\omega _n }}\alpha _n } }}{{2\pi \i z^{N/ \omega _n} }} \sum_{m(n)}  
\int_{\mathscr{C}^{(m)} (\theta _{nm}^ +)} {\frac{g(t)}{{(f(t) - f_n )^{(N + 1)/\omega _n} }}
\int_0^{\infty }\!\!\! {\frac{\e^{ - u} u^{\frac{{N + 1}}{{\omega _n }} - 1}}{1 + \left( \frac{u\e^{\pi \i(2\alpha _n -\omega _n) } }{z(f(t) - f_n )} \right)^{1/\omega _n} }\d u} \d t} , \label{remalt}
\end{split}\end{gather}
where $\mathscr{C}^{(m)} (\theta _{nm}^ +) := \mathscr{P}^{(m)} (\theta _{nm}^ +  ,\alpha _{nm}^ +  )\cup -\mathscr{P}^{(m)} (\theta _{nm}^ +  ,\alpha _{nm}^ +  +1)$. We note that
\begin{align*}
\arg \left(\frac{u\e^{\pi \i(2\alpha _n -\omega _n) } }{z(f(t) - f_n )}\right) &= 2\pi \alpha _n  - \pi \omega _n  - \theta - ( - \theta _{nm}^ +   + 2\pi \alpha _n )\\ 
&=  - \pi \omega _n  - \theta + \theta _{nm}^ +   >  - \pi \omega _n ,
\end{align*}
and
\begin{align*}
\arg \left(\frac{u\e^{\pi \i(2\alpha _n -\omega _n) } }{z(f(t) - f_n )}\right) &= 2\pi \alpha _n  - \pi \omega _n  - \theta - ( - \theta _{nm}^ +   + 2\pi \alpha _n )=  - \pi \omega _n  - \theta + \theta _{nm}^ +   \\ 
&=  - \pi \omega _n  - \theta + \theta _{nm}^ -   + 2\pi \omega _n  =  \pi \omega _n  - \theta + \theta _{nm}^ -  <\pi \omega_n ,
\end{align*}
whenever $t\in \mathscr{C}^{(m)} (\theta _{nm}^ +)$. Thus, $$\left|\arg \left(\frac{u\e^{\pi \i(2\alpha _n -\omega _n) } }{z(f(t) - f_n )}\right)\right| < \pi \omega _n.$$ Consequently, the $u$-integral may be expressed in terms of the generalised first-level hyperterminant as
\begin{multline*}
\int_0^{\infty } {\frac{\e^{ - u} u^{\frac{{N + 1}}{{\omega _n }} - 1}  }{1 + \left( \frac{u\e^{\pi \i(2\alpha _n -\omega _n) }}{z(f(t) - f_n )} \right)^{1/\omega _n} }\d u} \\ 
= \e^{ - \pi \frac{{N + 1}}{{\omega _n }}\i} \left(\e^{ \pi \i(\omega _n-2\alpha_n) } z(f(t) - f_n )\right)^{\frac{1}{\omega _n} } 
\HypTermOne{\e^{ \pi \i(\omega _n-2\alpha_n) } z(f(t) - f_n )}{\frac{{N + 1}}{{\omega _n }}}{\omega_n}{1}
.
\end{multline*}
Inserting this expression into \eqref{remalt}, we obtain the following alternative representation of $R_N^{(n)} (z;\alpha _n )$:
\begin{gather}\label{altrep}
\begin{split}
&R_N^{(n)} (z;\alpha _n ) = \frac{\e^{(2\alpha_n-1)\pi \i\frac{N + 1}{\omega _n } } }{2\pi \i z^{N/\omega _n} } \sum_{m(n)}  
\int_{\mathscr{C}^{(m)} (\theta _{nm}^ +)} \frac{g(t)}{(f(t) - f_n )^{(N + 1)/\omega _n} } \\ 
&\quad \times \left(\e^{ \pi \i(\omega _n-2\alpha_n) } z(f(t) - f_n )\right)^{\frac{1}{\omega _n} } 
\HypTermOne{\e^{ \pi \i(\omega _n-2\alpha_n) } z(f(t) - f_n )}{\frac{{N + 1}}{{\omega _n }}}{\omega_n}{1}\d t .
\end{split}
\end{gather}
This representation is valid when $\theta_{nm_1}^- -\frac{\pi}{2}<\theta <\theta_{nm_2}^+ +\frac{\pi}{2}$ (cf. \eqref{bound} below). We may then bound the $t$ integral as follows
\begin{align*}
\left| R_N^{(n)} (z;\alpha _n ) \right|
& \le  \frac{\Gamma \left( \frac{N + 1}{\omega _n} \right)}{2\pi \left| z \right|^{N/\omega _n} } 
 \sum_{m(n)}  \int_{\mathscr{C}^{(m)} (\theta _{nm}^ +)} \left| \frac{g(t)}{(f(t) - f_n )^{(N + 1)/\omega _n} }\d t \right|\\
& \times \mathop {\sup }\limits_{r \ge 1} \left| \frac{\left(z\big| \mathcal{F}_{nm}^ +\big|\e^{(\pi \omega _n-\theta _{nm}^ +  )\i} r\right)^{\frac{1}{\omega _n }}}{\Gamma \left( \frac{N + 1}{\omega _n } \right)}
\HypTermOne{z\big| \mathcal{F}_{nm}^ +\big|\e^{(\pi \omega _n-\theta _{nm}^ +   )\i} r}{\frac{{N + 1}}{{\omega _n }}}{\omega_n}{1}  \right|.
\end{align*}
A further simplification of this bound is possible, by employing the estimates for the generalised first-level hyperterminant given in Appendix \ref{appB}. In this way, we obtain
\begin{gather}\label{bound}
\begin{split}
& \left| R_N^{(n)} (z;\alpha _n ) \right| 
 \le \frac{\Gamma \left( \frac{N + 1}{\omega _n} \right)}{2\pi \left| z \right|^{N/\omega _n} } \sum_{m(n)}  \int_{\mathscr{C}^{(m)} (\theta _{nm}^ +)} \left| \frac{g(t)}{(f(t) - f_n )^{(N + 1)/\omega _n} }\d t \right|\times\\ 
 & 
 \begin{cases} 1 & \!\!\!\text{ if } \;  |\theta - \theta _{nm}^ +   + \pi \omega _n | \leq \frac{\pi}{2}\omega_n , \\ 
 \min\left(\left| \csc \left( \frac{\theta - \theta _{nm}^ +  }{\omega _n } \right) \right|,\omega _n \sqrt {\e\left( \frac{N + 1}{\omega _n} + \frac{1}{2} \right)} \right) & 
 \!\!\!\text{ if } \; \frac{\pi}{2}\omega_n  < |\theta - \theta _{nm}^ +   + \pi \omega _n | \leq \pi \omega_n , \\ 
 \frac{\sqrt {2\pi \omega _n (N + 1)} }{\left| \cos (\theta - \theta _{nm}^ +  ) \right|^{\frac{N + 1}{\omega _n }} } + \omega _n \sqrt {\e\left( \frac{N + 1}{\omega _n} + \frac{1}{2} \right)} & 
 \!\!\!\!\!\!\!\!\!\text{ if } \; \pi \omega_n  < |\theta - \theta _{nm}^ +   + \pi \omega _n| < \pi \omega_n+\frac{\pi}{2}. \end{cases}
\end{split}
\end{gather}
In the case of linear endpoint ($\omega_n=1$), the quantity $\sqrt {\e\left( N + \frac{3}{2} \right)}$ in \eqref{bound} can be replaced by \eqref{gammaratio} with $M=N+1$.

In \eqref{eq101} we may expand the loop contour of integration around the critical point $t^{(n)}$ across the domain $\Delta^{(n)}$ to obtain a representation of the asymptotic coefficients in terms of integrals over the contours $\mathscr{C}^{(m)} (\theta _{nm}^ +)$ as follows,
\begin{equation}\label{neglectedTerm}
\left| {\frac{{T_N^{(n)} (\alpha _n )}}{{z^{N/\omega _n} }}} \right| 
= \frac{\Gamma \left( \frac{N + 1}{\omega _n} \right)}{{2\pi \left| z \right|^{N/\omega _n} }} \left| \sum_{m(n)} \int_{\mathscr{C}^{(m)} (\theta _{nm}^ +)} \frac{{g(t)}}{{(f(t) - f_n )^{(N + 1)/\omega _n} }}\d t   \right|.
\end{equation}
This representation illustrates the close relation between the form of the bound \eqref{bound} and the absolute value of the first neglected 
term.  The modulus bars are inside the integral in \eqref{bound} whereas they are at the outside of the integral in \eqref{neglectedTerm}.
However, in \eqref{neglectedTerm} we integrate along steepest descent paths $\mathscr{C}^{(m)} (\theta _{nm}^ +)$ on which $f(t)-f_{n}$ is 
monotonically decreasing. This means that only when $g(t)$ is highly oscillatory, will the integral in \eqref{bound} be considerably larger  than the integral in \eqref{neglectedTerm}. 
The larger the value of $N$, the smaller the difference in size of the two integrals.

Figure \ref{figure4}, for our first example below, clearly demonstrates the asymptotic property that sizes of the exact terms and the 
corresponding remainders are approximately the same.   This follows from the factor $1$ in the second line
of \eqref{bound}. In Figure \ref{figure9}, which is for our second example, the remainders are considerably larger than the terms. That
example illustrates the effect of the additional factor $\omega _n \sqrt {\e\left( \frac{N + 1}{\omega _n} + \frac{1}{2} \right)}$ in the third line of 
\eqref{bound} pertaining to the parameters $\theta$, $\omega_{n}$ and $\theta_{nm}^{+}$ of that particular calculation.

\subsection{Bounds for simple saddles} 
If $t^{(n)}$ is a simple saddle, then the integral over the double infinite contour through $t^{(n)}$ can be expanded as
\[
{\bf T}^{(n)}(z,0)= \sum_{r=0}^{N-1}\frac{{\bf T}_{2r}^{(n)}(0)}{z^r} + {\bf R}_N^{(n)}(z,0),
\]
with ${\bf R}_N^{(n)}(z,0)= R_{2N}^{(n)}(z;0)-R_{2N}^{(n)}(z;1)$. The estimation of ${\bf R}_N^{(n)}(z,0)$ was considered by Boyd \cite{Boyd93} in the case that all the adjacent saddles are simple. Employing \eqref{altrep} and simplifying the result, we obtain
\begin{align*}
{\bf R}_N^{(n)}(z,0) = \frac{( - 1)^{N + 1} }{\pi z^N}\sum_{m(n)} &\int_{\mathscr{C}^{(m)} (\theta _{nm}^ +)} \frac{g(t)}{(f(t) - f_n )^{N + \frac{1}{2}} }\\
&\times\e^{\pi \i} z(f(t) - f_n )F^{(1)} \left( \e^{\pi \i} z(f(t) - f_n );\begin{array}{c}
   {N + \frac{1}{2}}  \\
   1  \\
\end{array} \right)\d t  .
\end{align*}
This representation is valid when $\theta_{nm_1}^- -\frac{\pi}{2}<\theta <\theta_{nm_2}^+ +\frac{\pi}{2}$. We may then bound the $t$ integral as follows
\begin{align*}
\left|{\bf R}_N^{(n)}(z,0)\right| \le \; & \frac{\Gamma \left( N + \frac{1}{2} \right)}{\pi \left| z \right|^N}\sum_{m(n)} \int_{\mathscr{C}^{(m)} (\theta _{nm}^ +)} \left| \frac{g(t)}{(f(t) - f_n )^{N + \frac{1}{2}} }\d t \right| \\ &\times \mathop {\sup }\limits_{r \ge 1} \left| \frac{z|\mathcal{F}_{nm}^ +  |\e^{(\pi  - \theta _{nm}^ +  )\i} r}{\Gamma \left( N + \frac{1}{2} \right)}F^{(1)} \left( z|\mathcal{F}_{nm}^ +  |\e^{(\pi  - \theta _{nm}^ +  )\i} r;\begin{array}{c}
   {N + \frac{1}{2}}  \\
   1  \\
\end{array} \right) \right| .
\end{align*}
A further simplification of this bound is possible, by applying the estimates for the generalised first-level hyperterminant given in Appendix \ref{appB}. In this way, we deduce
\begin{align}\label{bounder}
\left|{\bf R}_N^{(n)}(z,0)\right| \le \; & \frac{\Gamma \left( N + \frac{1}{2} \right)}{\pi \left| z \right|^N}\sum_{m(n)} \int_{\mathscr{C}^{(m)} (\theta _{nm}^ +)} \left| \frac{g(t)}{(f(t) - f_n )^{N + \frac{1}{2}} }\d t \right| \\ \nonumber
\\ \nonumber & \times \begin{cases} 1 & \text{ if } \;  |\theta - \theta _{nm}^ +   + \pi | \leq \frac{\pi}{2} , \\ \min(| \csc ( \theta - \theta _{nm}^ +  )|, \sqrt {\e( N+1)} ) & \text{ if } \; \frac{\pi}{2}  < |\theta - \theta _{nm}^ +   + \pi  | \leq \pi , \\  \nonumber \frac{\sqrt {2\pi \left(N + \frac{1}{2}\right)} }{\left| \cos (\theta - \theta _{nm}^ +  ) \right|^{N+\frac{1}{2}} } + \sqrt {\e ( N + 1)} & \text{ if } \; \pi  < |\theta - \theta _{nm}^ +   + \pi | < \frac{3\pi}{2}. \end{cases}
\end{align}
The quantity $\sqrt {\e(N+1)}$ in this bound can be replaced by \eqref{gammaratio} with $M=N+\frac{1}{2}$. 

The bound (\ref{bounder}) improves Boyd's \cite{Boyd93} results in three ways.  First, it is more general in that the adjacent saddles need not to be simple.  
Second, (\ref{bounder}) extends the range of validity of the bound to include $\pi  < |\theta - \theta _{nm}^ +   + \pi | < \frac{3\pi}{2}$.  Third, 
the new result sharpens the bound with a factor $ \sqrt {\e ( N + 1)}$ in place of Boyd's larger $2\sqrt{N}$ factor, and for this larger factor to hold
he even requires the extra assumption $N \ge \cot ^2 \left( \frac{1}{2}\left( \theta _{nm_2 }^ +   - \theta _{nm_1 }^ -  \right) \right)$.

\section{Example 1: Pearcey on the cusp}\label{sec6}

A rescaled Pearcey function (compare \cite[\href{http://dlmf.nist.gov/36.2}{\S36.2}]{NIST:DLMF}) is defined by the integral 
\begin{equation}  \label{eq24a}
\Psi_{2}(x,y;z)=\int_{-\infty}^{+\infty}\e^{-zf(t;x,y)} \d t, \qquad f(t;x,y)=-\i\left(t^{4}+yt^{2}+xt\right).
\end{equation}
Due to the polynomial nature of the exponent function and the ability to scale $t$, $z$, with $x$ and $y$, without loss of generality the 
modulus of the large parameter $z$ may be set to $1$.  The function represents the wavefield in the neighbourhood of the canonically stable cusp catastrophe \cite{Berry69} and occurs commonly in two dimensional linear wave problems. 

The integrand possesses three saddle points $t^{(j)}, j=1,2,3$, satisfying 
\begin{equation} \nonumber 
f'(t^{(j)};x,y)=4\left(t^{(j)}\right)^{3}+2yt^{(j)}+x=0.
\end{equation}
In \cite{BH91} a hyperasymptotic expansion of the Pearcey function was calculated in the case of three distinct saddle points. Here we have 
extended that analysis to cover the case where two of the saddles have coalesced.

Two of the three saddle points coalesce on the cusp-shaped caustic given by
\begin{equation} \nonumber 
f'(t;x,y)=f''(t;x,y)=0 \qquad \Rightarrow \qquad 27x^{2}=-8y^{3}, \qquad (x,y)\ne 0,
\end{equation}
see Figure \ref{figure3ab}(a).
(At the origin $(x,y)=(0,0)$, all three saddles coalesce, where the integral reduces to an exact explicit representation  \cite[\href{http://dlmf.nist.gov/36.2.E15}{36.2.15}]{NIST:DLMF}.)

We shall choose $x=2\sqrt{2}$, $y=-3$.  There is a simple saddle at $t^{(1)}=-\sqrt{2}$ and a double saddle denoted by $t^{(2)}=1/\sqrt{2}$.  The asymptotic expansion about $t^{(1)}$ 
has $\omega_{1}=2$ and is controlled by the double saddle at $t^{(2)}$ with $\omega_{2}=3$, and vice versa.  


\begin{figure}

\subfigure[ \hspace{0mm} ]{\includegraphics[width=0.4\hsize]{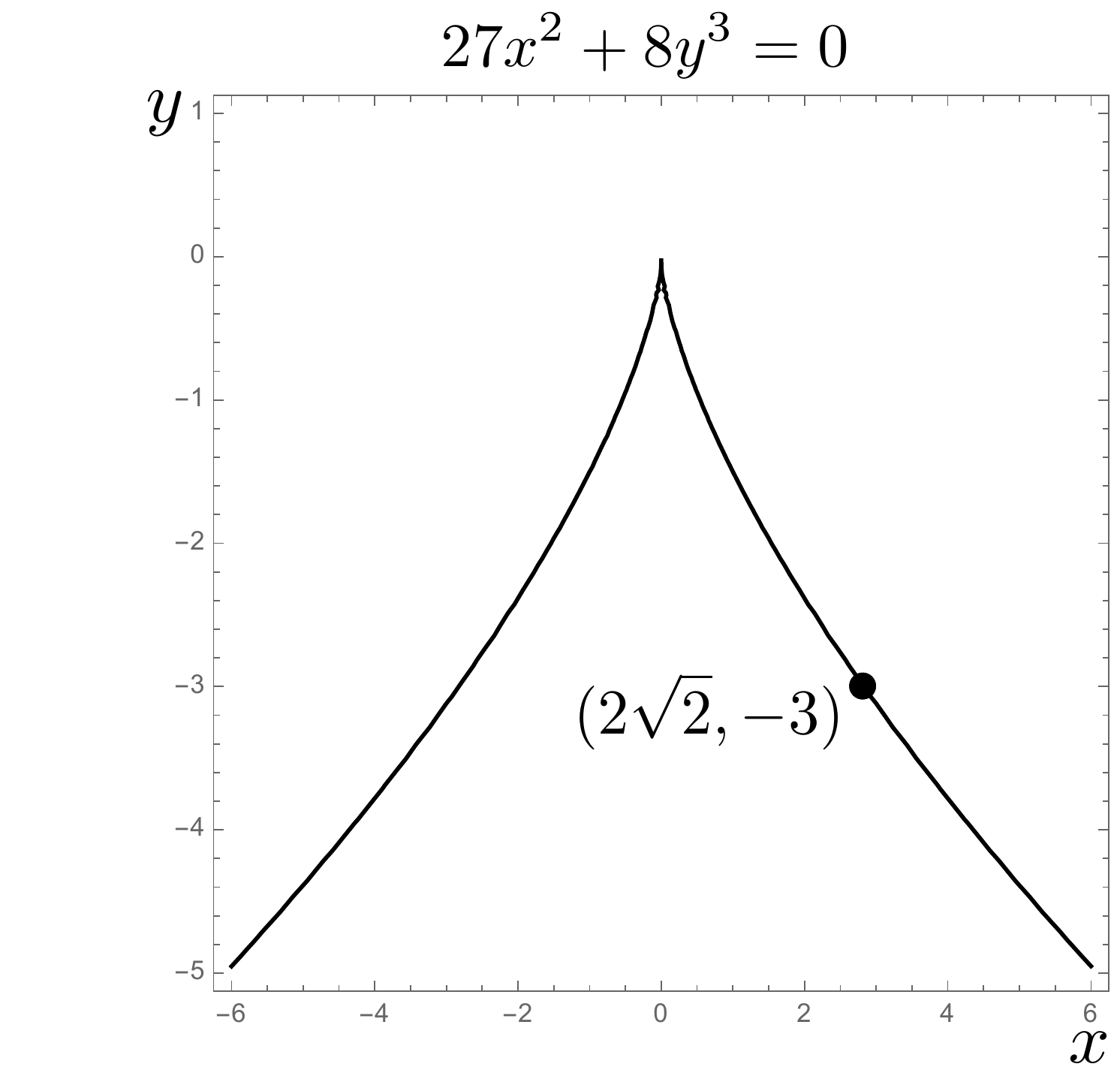}}
\hfill
\subfigure[ \hspace{0 mm} ]{\includegraphics[width=0.4\hsize]{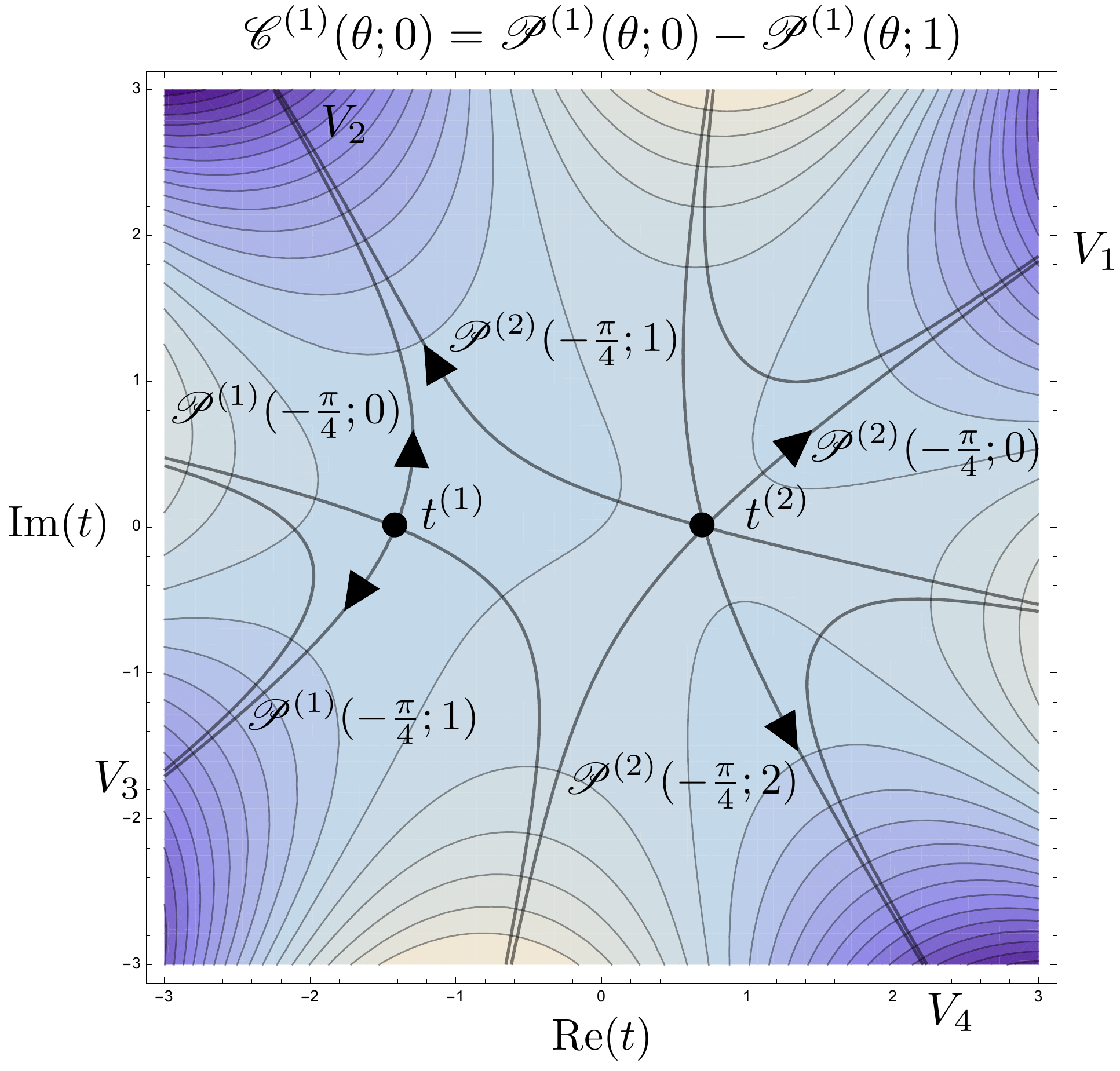}}
\hfill
\subfigure[\hspace{0 mm} ]{\includegraphics[width=0.4\hsize]{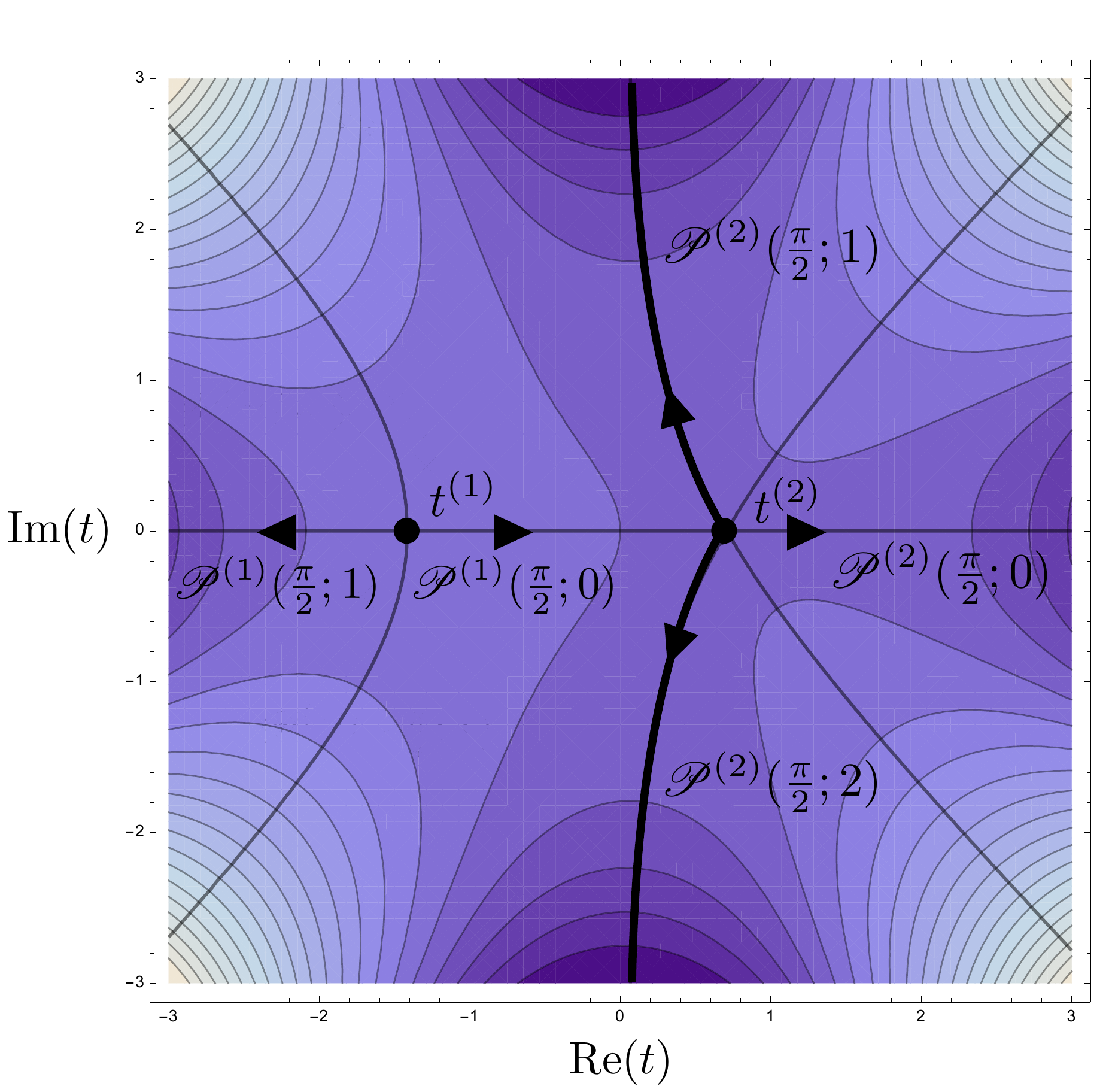}}
\hfill
\subfigure[\hspace{0 mm} ]{\includegraphics[width=0.4\hsize]{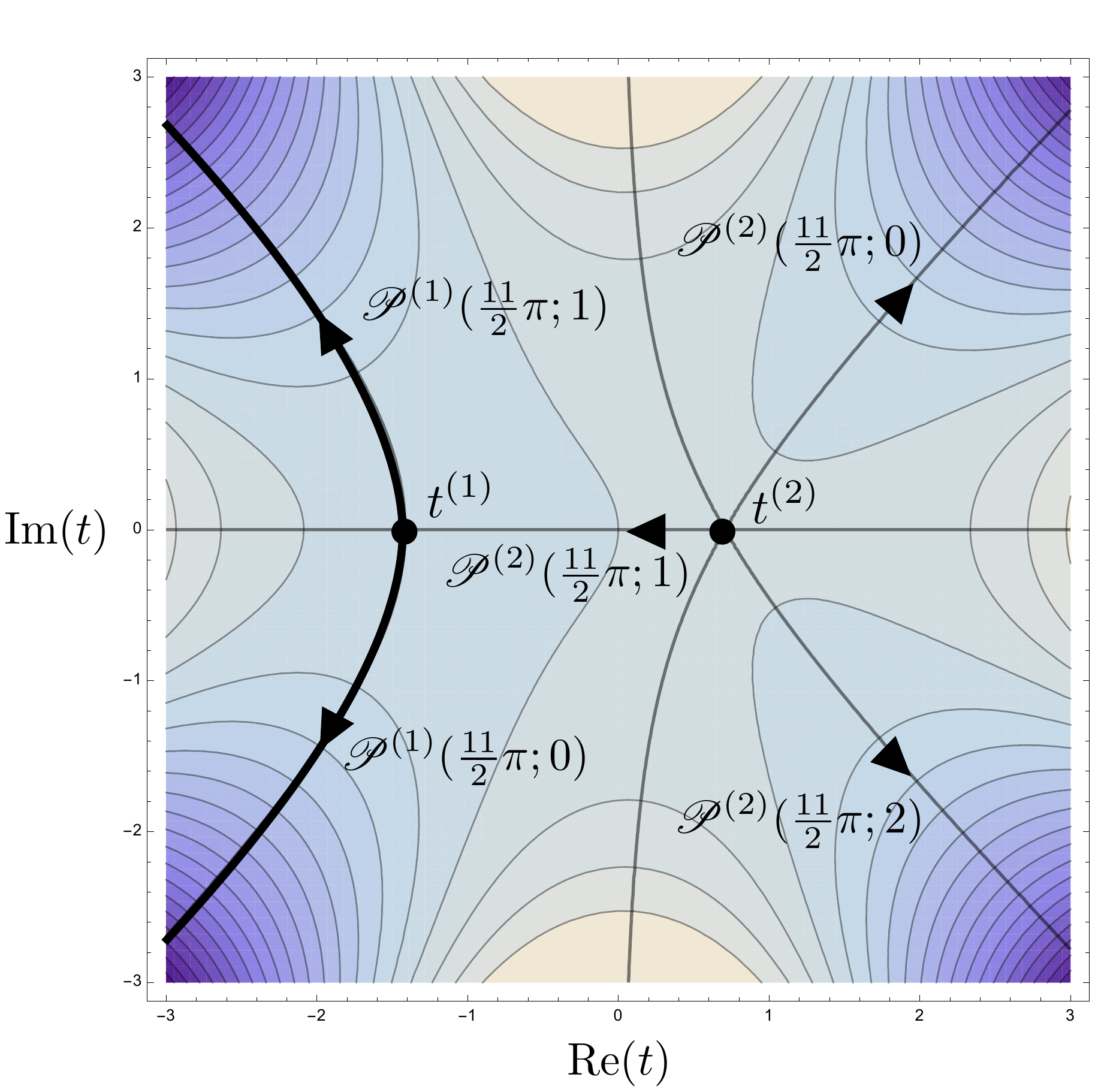}}

\caption{(a) Location of the parameter point $(x,y)=(2\sqrt{2},-3)$ at which we evaluate the integral (\ref{eq24a}) relative to the caustic of the Pearcey function, satisfying $27x^{2}
=-8y^{3}$. (b) The steepest descent paths ${\mathscr P}^{(1)}(-\frac{\pi}{4},0)$, ${\mathscr P}^{(1)}(-\frac{\pi}{4},1)$ in the complex $t$-plane emerging from the simple saddle $t^{(1)}$ 
($\omega_{1}=2$) and travelling to labelled valleys $V_{j}$, $j=2, 3$ at infinity.  Also shown is the degenerate saddle $t^{(2)}$ ($\omega_{2}=3$).  (c) The steepest descent paths $
{\mathscr P}^{(2)}(\frac{\pi}{2},\alpha_{2})$, $\alpha_{2}=0,1,2$, emerging from $t^{(2)}$, as a Stokes phenomenon occurs between $t^{(1)}$ and $t^{(2)}$ when $\theta_{12}^{+}=\frac{\pi}{2}$. 
The bold lines are the steepest paths that are used in the Level 1 hyperasymptotic expansion about $t^{(1)}$ (\ref{Step3}), (\ref{double}). (d)  The steepest descent paths ${\mathscr 
P}^{(2)}(\frac{11}{2}\pi,\alpha_{2})$, $\alpha_{2}=0,1,2$, emerging from $t^{(2)}$, as a Stokes phenomenon occurs between 
$t^{(2)}$ and $t^{(1)}$ when $\theta_{121}^{+}=\frac{11}{2}\pi$. The bold lines are the steepest paths that are used in the Level 2 hyperasymptotic expansion about $t^{(1)}$  (\ref{Step5}), 
(\ref{double}). (Or Level 1 hyperasymptotic expansion about $t^{(2)}$.)  }
\label{figure3ab}
\end{figure}


We shall calculate a hyperasymptotic expansion about $t^{(1)}$.  We take $z=\e^{\i\theta}$ and chose $\theta=-\frac{\pi}{4}$.  The steepest paths are denoted by
${\mathscr P}^{(1)}(-\frac{\pi}{4},0)$ and ${\mathscr P}^{(1)}(-\frac{\pi}{4},1)$, see Figure \ref{figure3ab}(b).

In the calculations below we will use \eqref{alphanm} many times and observe that in this case $ \arg (f^{(\omega _1 )} (t^{(1)} ))= \arg (f^{(\omega _2 )} (t^{(2)} ))=-\frac{\pi}{2}$, and in Figures \ref{figure3ab}(c,d)
for the curve ${\mathscr P}^{(2)}(\frac{\pi}{2},1)$ we have $\varphi=\frac{2}{3}\pi$ and for curve ${\mathscr P}^{(1)}(\frac{11}{2}\pi,1)$ we have $\varphi=-\frac{\pi}{2}$.


The normalised integrals that we consider are 
\begin{equation} \nonumber 
T^{(1)} (z;\alpha_1) = 2 z^{1/2} \int_{{\mathscr P}^{(1)}(-\frac{\pi}{4},\alpha_1)} {\e^{z\i(t^4  - 3t^2  + 2\sqrt 2 t + 6)}\d t} ,\qquad \alpha_1=0,1,
\end{equation}
which posses the asymptotic expansions
\begin{equation} \label{eq27b}
T^{(1)} (z;\alpha_1) = \sum\limits_{r = 0}^{N_0^{(1)} - 1} {\frac{{T_{r}^{\left( 1 \right)} \left( {\alpha_1} \right)}}{{z^{r/2} }}}  
+ R_1^{(1)}(z;\alpha_1),
\end{equation}
with coefficients
\begin{gather}\label{Pcoeff}
\begin{split}
T_{r}^{\left( 1 \right)} \left( {0} \right) & = \e^{\frac{\pi }{4}(r + 1)\i}  \frac{{\Gamma ( \frac{r+1}{2})}}{{\Gamma (r + 1)}}\left[ {\frac{{\d^{r} }}{{\d t^{r} }}\left( {\frac{{(t + \sqrt 2 )^2 }}{{t^4  - 3t^2  + 2\sqrt 2 t + 6}}} 
\right)^{(r +1)/2} } \right]_{t =  - \sqrt 2 } \\  
& =\e^{\frac{\pi }{4}(r + 1)\i}  \frac{{\Gamma ( \frac{r+1}{2})}}{{\Gamma (r + 1)}}\left[ {\frac{{\d^{r} }}{{\d t^{r} }}\left( {\frac{1}{{t^2  - 4\sqrt 2 t + 9}}} \right)^{(r+1)/2} } \right]_{t = 0} 
\\ & = \frac{\e^{\frac{\pi }{4}(r + 1)\i}}{{3^{2r + 1} }}\Gamma \Big(\frac{r+1}{2}\Big)C_{r}^{(\frac{r+1}{2})} \Big(\frac{{2\sqrt 2 }}{3}\Big),
\end{split}
\end{gather}
 $T_{r}^{\left( 1 \right)} \left( {\alpha_1} \right)=\e^{2\pi\i\alpha_1 (r+1)/2}T_{r}^{\left( 1 \right)} \left( {0} \right)$.
In deriving the coefficients, in the penultimate line of (\ref{Pcoeff}) we have recognised the presence of the generating function \cite[\href{http://dlmf.nist.gov/18.12.E4}{eq.~18.12.4}]{NIST:DLMF} for the ultraspherical polynomials $C_r^{\left( p \right)}(w)$ .

We will also need the coefficients of the asymptotic expansions of the integrals
\begin{equation}  \nonumber 
T^{(2)} (z;\alpha_2) = 3 z^{1/3} \int_{{\mathscr P}^{(2)}(-\frac{\pi}{4},\alpha_2)} {\e^{z\i(t^4  - 3t^2  + 2\sqrt 2 t -\frac34)}\d t} ,\qquad \alpha_2=0,1,2,
\end{equation}
which possess the asymptotic expansions
\begin{equation}  \nonumber 
T^{(2)} (z;\alpha_2) = \sum\limits_{r = 0}^{N_0^{(1)} - 1} {\frac{{T_{r}^{\left( 2 \right)} \left( {\alpha_2} \right)}}{{z^{r/3} }}}  
+ R_0^{(2)}(z;\alpha_2),
\end{equation}
with coefficients
\begin{align*}
T_{r}^{\left( 2 \right)} \left( {0} \right) & = \e^{\frac{\pi }{6}(r + 1)\i}  \frac{{\Gamma ( \frac{r+1}{3})}}{{\Gamma (r + 1)}}\left[ {\frac{{\d^{r} }}{{\d t^{r} }}\left( {\frac{{(t - 1/\sqrt 2 )^3 }}{{t^4  - 3t^2  + 2\sqrt 2 t -3/4}}} 
\right)^{(r +1)/3} } \right]_{t =  1/\sqrt 2 } \\
& =\e^{\frac{\pi }{6}(r + 1)\i}  \frac{{\Gamma ( \frac{r+1}{3})}}{{\Gamma (r + 1)}}\left[ {\frac{{\d^{r} }}{{\d t^{r} }}\left( {\frac{1}{{t+2\sqrt 2}}} \right)^{(r+1)/3} } \right]_{t = 0} 
\\ & = \frac{\e^{\frac{\pi }{6}(r + 1)\i}}{{2^{2r + 1/2} }}\Gamma \Big(\frac{r+1}{3}\Big){-\frac{r+1}{3}\choose r},
\end{align*}
and $T_{r}^{\left( 2 \right)} \left( {\alpha_2} \right)=\e^{2\pi\i\alpha_2 (r+1)/3}T_{r}^{\left( 2 \right)} \left( {0} \right)$.


For the singulant on the caustic we have
\begin{equation}\nonumber 
\left|\mathcal{F}_{12}^{+}\right|=\left|f(t^{(2)}; 2\sqrt{2},-3)-f(t^{(1)}; 2\sqrt{2},-3)\right|=\frac{27}{4}.
\end{equation}
The effective asymptotic parameter in the expansion is thus $\left|z\mathcal{F}_{12}^{+}\right|=6.75$, and hence, the optimal number of 
terms in \eqref{eq27b} is $N_0^{(1)}=\left[\left|z\mathcal{F}_{12}^{+}\right|\omega_1\right]=13$.

Since $\theta=-\frac{\pi}{4}$ it follows that for the integral $T^{(1)}(z;0)$, the corresponding $\theta_{12}^{+}=\frac{\pi}{2}$. The corresponding contour of integration emanating from adjacent saddle $t^{(2)}$
is ${\mathscr P}^{(2)}(\frac{\pi}{2},1)$, see Figure \ref{figure3ab}(c), and hence, the Level 1 re-expansion is of the form
\begin{equation}\nonumber 
R_{0}^{(1)}(z;\alpha_n)=
\frac{z^{(1-N_0^{(1)})/2}}{6\pi\i} \sum_{r=0}^{N_1^{(2)}-1}{\bf T}_r^{(2)}(1)
\HypTermOne{z}{\frac{N_0^{(1)}+1}{2}-\frac{r+1}{3}}{2}{\frac{27}{4}\e^{\frac{\pi}{2}\i}}+R_{1}^{(1)}(z;0).
\end{equation}
The optimal numbers of terms at Level 1 are  $N_0^{(1)}=\left[2\left|z\mathcal{F}_{12}^{+}\right|\omega_1\right]=27$ and $N_1^{(2)}=\left[\left|z\mathcal{F}_{12}^{+}\right|\omega_2\right]=20$.

With $\theta_{12}^{+}=\frac{\pi}{2}$ and contour ${\mathscr P}^{(2)}(\frac{\pi}{2},1)$ it follows that $\theta_{121}^{+}=\theta_{12}^{+}+5\pi=\frac{11}{2}\pi$, 
and the corresponding contour of integration emanating from adjacent saddle $t^{(1)}$
is ${\mathscr P}^{(1)}(\frac{11}{2}\pi,2)$, see Figure \ref{figure3ab}(d), and hence, the Level 2 re-expansion is of the form
\begin{gather}\begin{split}\nonumber 
R_{1}^{(1)}(z;0)=&\sum_{r=0}^{N_2^{(1)}-1}\frac{z^{(1-N_0^{(1)})/2}}{\left(2\pi\i\right)^2 6}\\
& \times\left({\bf T}_r^{(1)}(2)
\HypTermTwo{z}{\frac{N_0^{(1)}+1}{2}-\frac{N_1^{(2)}}{3}}{\frac{N_1^{(2)}+1}{3}-\frac{r+1}{2}}{2}{3}
{\frac{27}{4}\e^{\frac{\pi}{2}\i}}{\frac{27}{4}\e^{-\frac92 \pi\i}}\right.\\
&\qquad \left.-{\bf T}_r^{(1)}(3)
\HypTermTwo{z}{\frac{N_0^{(1)}+1}{2}-\frac{N_1^{(2)}}{3}}{\frac{N_1^{(2)}+1}{3}-\frac{r+1}{2}}{2}{3}
{\frac{27}{4}\e^{\frac{\pi}{2}\i}}{\frac{27}{4}\e^{-\frac{13}2 \pi\i}}\right)\\
&\quad+R_{2}^{(1)}(z;0).
\end{split}\end{gather}
The optimal numbers of terms at Level 2 are given in Table \ref{table1}.

Finally, with $\theta_{121}^{+}=\frac{11}2\pi$ and contour ${\mathscr P}^{(1)}(\frac{11}{2}\pi,2)$ it follows that $\theta_{1212}^{+}=\theta_{121}^{+}+3\pi=\frac{17}{2}\pi$,
$\alpha_{1212}^+=5$, and the optimal numbers in \eqref{Step6} are  again given in Table \ref{table1}.

\begin{table}[h]
\caption{The numbers of terms in each series of the hyperasymptotic expansion that are required to  minimise overall the absolute error for 
the $(1 \rightarrow 2)$ Pearcey example derived from (\ref{optimal}). Note that each row corresponds to a decision to stop the re-expansion 
at that stage.  Hence the table row corresponding to level ``two'' corresponds to the truncations required at each level up to two, after 
deciding to stop after two re-expansions of the remainder. Note that all the truncations change with the decision to stop at a particular level.}
\begin{center}
\begin{tabular}{|c|c|c|c|c|c|}
\hline
Level & $N_0^{(1)}$ & $N_1^{(2)}$ & $N_2^{(1)}$ & $N_3^{(2)}$ & error \\
\hline
zero & 13 &  &  &  & $1.9\times 10^{-4}$ \\
one & 27 & 20 &  &  & $9.5\times 10^{-9}$ \\
two & 40 & 40 & 13 &  & $3.8\times 10^{-14}$ \\
three & 54 & 60 & 27 & 20 & $9.0\times 10^{-17}$ \\
\hline
\end{tabular}
\end{center}
\label{table1}
\end{table}%

When we compute our integral numerically with high precision for these values of $x$, $y$ and $z$ we obtain
$$T^{(1)}(z,0)=
0.37277007370182291370+
0.47493131741141216950\i.
$$
The numerics of the hyperasymptotic approximations are given in Table \ref{table1}, and for the Level 3 expansion we display the terms and errors in Figure \ref{figure4}.
We observe in this figure that the remainders in the original Poincar\'e expansions are of the same size as the first neglected terms, as predicted in Section \ref{sec5}.
In fact at all levels are the remainders of a similar size than the first neglected terms. Occasionally, the remainders are considerably smaller.

In this section we derived hyperasymptotic approximations for $T^{(1)}(z,0)$.
Note that we can repeat the calculation for the integral $T^{(1)}(z,1)$.
The \emph{only} changes in the re-expansions are that all the $\theta^+$ are increased by $2\pi$ and all the $\alpha^+$ are increased by $1$.
The optimal numbers of terms will remain the same.

\begin{figure}
\centering\includegraphics[width=0.6\textwidth]{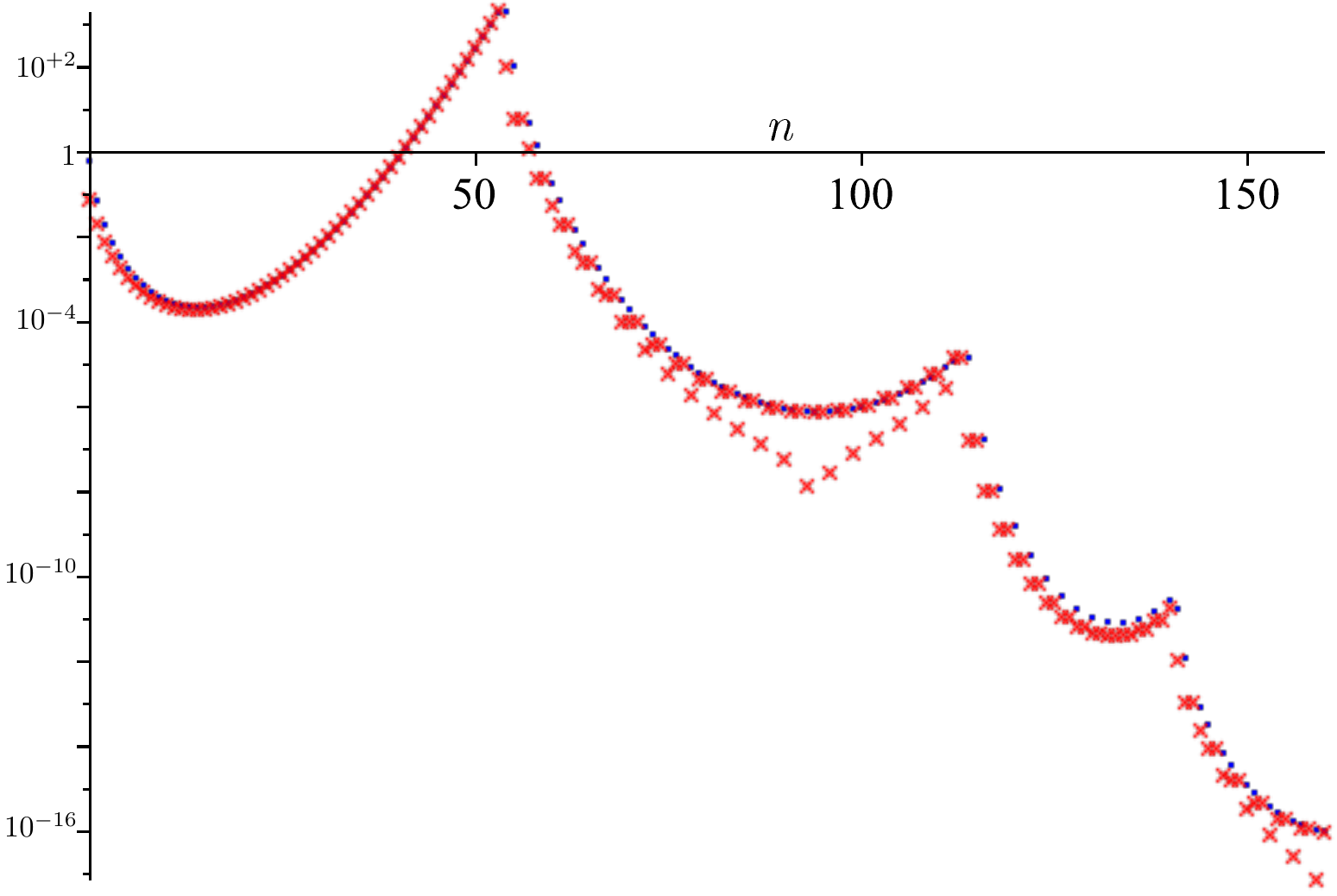}%
\caption{For example 1: The modulus of the $n^{\rm th}$ term in the Level 3 hyperasymptotic expansion (blue dots), and the modulus of the remainder after taking $n$ terms in the approximation (red crosses).}
\label{figure4}
\end{figure}

\section{Example 2: Higher order saddles}\label{sec7}

In the second main example we take an integral of the form \eqref{eq1}, but now with $g(t) \equiv 1$ and
$$f(t)=\ifrac{15}{28}t^7-5t^6+18t^5-30t^4+20t^3\qquad\Longrightarrow\qquad f'(t)=\ifrac{15}4t^2\left(t-2\right)^4.$$
The saddle points are $t^{(1)}=0$ and $t^{(2)}=2$, with  $\omega_1=3$ and $\omega_2=5$.  Hence this example is an example of the 
hyperasymptotic method when both saddles are degenerate. 

Once again, due to the scaling properties of the polynomial $f(t)$ we may take $z=\e^{\i\theta}$ and also choose $\theta=-\frac{\pi}{4}$. The steepest 
descent paths are displayed in Figure \ref{figure8}(a).
For the coefficients in the asymptotic expansions we have
$$T_{r}^{\left( 1 \right)} \left( {0} \right) = \frac{{\Gamma ( \frac{r+1}{3})}}{{\Gamma (r + 1)}}\left[ {\frac{{\d^{r} }}{{\d t^{r} }}
\left( {\frac{{1 }}{{\frac{15}{28}t^4-5t^3+18t^2-30t+20}}} \right)^{(r +1)/3} } \right]_{t =  0 },$$
\begin{align*}
T_{r}^{\left( 2 \right)} \left( {0} \right) & =\frac{{\Gamma ( \frac{r+1}{5})}}{{\Gamma (r + 1)}}\left[ {\frac{{\d^{r} }}{{\d t^{r} }}
\left( {\frac{{(t -2 )^5 }}{{\frac{15}{28}t^7-5t^6+18t^5-30t^4+20t^3-\frac{32}7}}} \right)^{(r +1)/5} } \right]_{t =  2 }\\
& =\frac{{\Gamma ( \frac{r+1}{5})}}{{\Gamma (r + 1)}}\left[ {\frac{{\d^{r} }}{{\d t^{r} }}
\left( {\frac{1}{{\frac{15}{28}t^2+\frac52t+3}}} \right)^{(r +1)/5} } \right]_{t =  0 }\\
& = \frac{\left(5/28\right)^{r/2}}{{3^{(r + 1)/5} }}\Gamma \Big(\frac{r+1}{5}\Big)C_{r}^{(\frac{r+1}{5})} \Big(-\sqrt{\frac{35}{36}}\Big),
\end{align*}
and the other coefficients are defined via $T_{r}^{\left( m \right)} \left( {\alpha_m} \right)=\e^{2\pi\i\alpha_m (r+1)/\omega_m}T_{r}^{\left( m \right)} \left( {0} \right)$.

For the singulant we have
\begin{equation}\nonumber 
\left|\mathcal{F}_{12}^{+}\right|=\left|f(2)-f(0)\right|=\frac{32}{7}.
\end{equation}
The effective asymptotic parameter in the expansion is thus $\left|z\mathcal{F}_{12}^{+}\right|=\frac{32}{7}$, and hence, the optimal number of 
terms in
\begin{equation}  \label{eq35b}
T^{(1)} (z;\alpha_1) = \sum\limits_{r = 0}^{N_0^{(1)} - 1} {\frac{{T_{r}^{\left( 1 \right)} \left( {\alpha_1} \right)}}{{z^{r/3} }}}  
+ R_1^{(1)}(z;\alpha_1),
\end{equation}
is $N_0^{(1)}=\left[\left|z\mathcal{F}_{12}^{+}\right|\omega_1\right]=13$. 

We will focus again on $T^{(1)} (z;0)$ and give only the main details, which are,
$$\theta_{12}^+=0,\quad \theta_{121}^+=9\pi,\quad \theta_{1212}^+=14\pi,\quad \alpha_{12}^+=2,\quad \alpha_{121}^+=4,\quad \alpha_{1212}^+=9.$$
When we compute this integral numerically for this value of $z$ with high precision, we obtain
$$T^{(1)}(z,0)=
1.244081553113296+
0.145693991003805\i.
$$
The numerics of the hyperasymptotic approximations are given in Table \ref{table2}, and for the Level 2 expansion we display the terms and errors in Figure \ref{figure9}.
We observe that this time the remainders in the original Poincar\'e expansion are considerably larger than the first neglected terms, again, as predicted in Section \ref{sec5}.
However, in the higher levels the remainders are again of a similar size than the first neglected terms.

\begin{table}[h]
\caption{The numbers of terms required to minimise the absolute error at each level of the hyperasymptotic re-expansions for the $(3 \rightarrow 5)$ degenerate example.}
\begin{center}
\begin{tabular}{|c|c|c|c|c|c|}
\hline
Level & $N_0^{(1)}$ & $N_1^{(2)}$ & $N_2^{(1)}$ & $N_3^{(1)}$ & error \\
\hline
zero & 13 &  &  &  & $6.9\times 10^{-3}$ \\
one & 27 & 22 &  &  & $3.7\times 10^{-7}$ \\
two & 41 & 45 & 13 &  & $2.0\times 10^{-10}$ \\
three & 54 & 68 & 27 & 22 & $1.1\times 10^{-13}$ \\
\hline
\end{tabular}
\end{center}
\label{table2}
\end{table}%


\begin{figure}

\subfigure[ \hspace{0 mm} ]{\includegraphics[width=0.32\hsize]{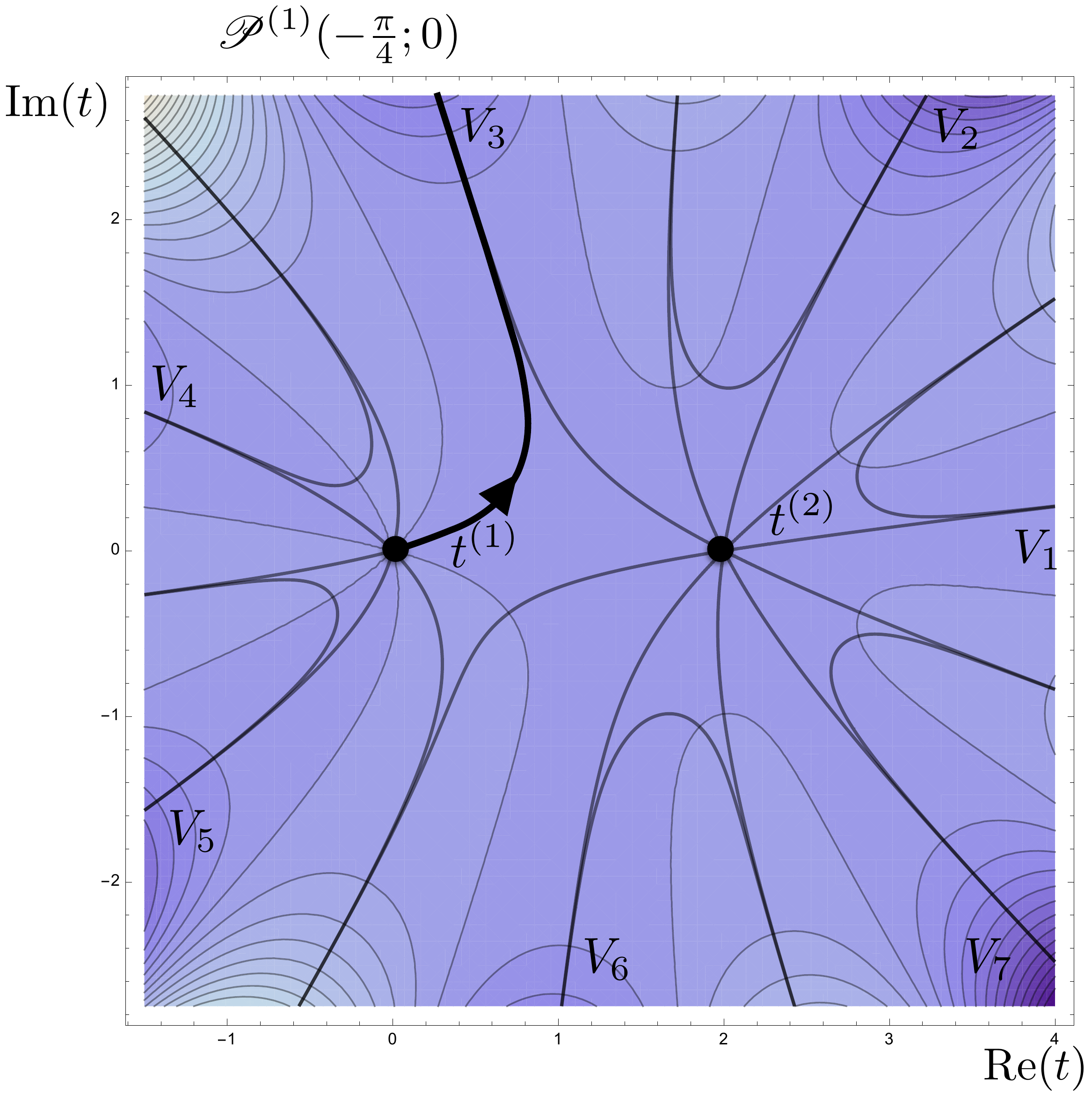}}
\subfigure[ \hspace{0 mm} ]{\includegraphics[width=0.32\hsize]{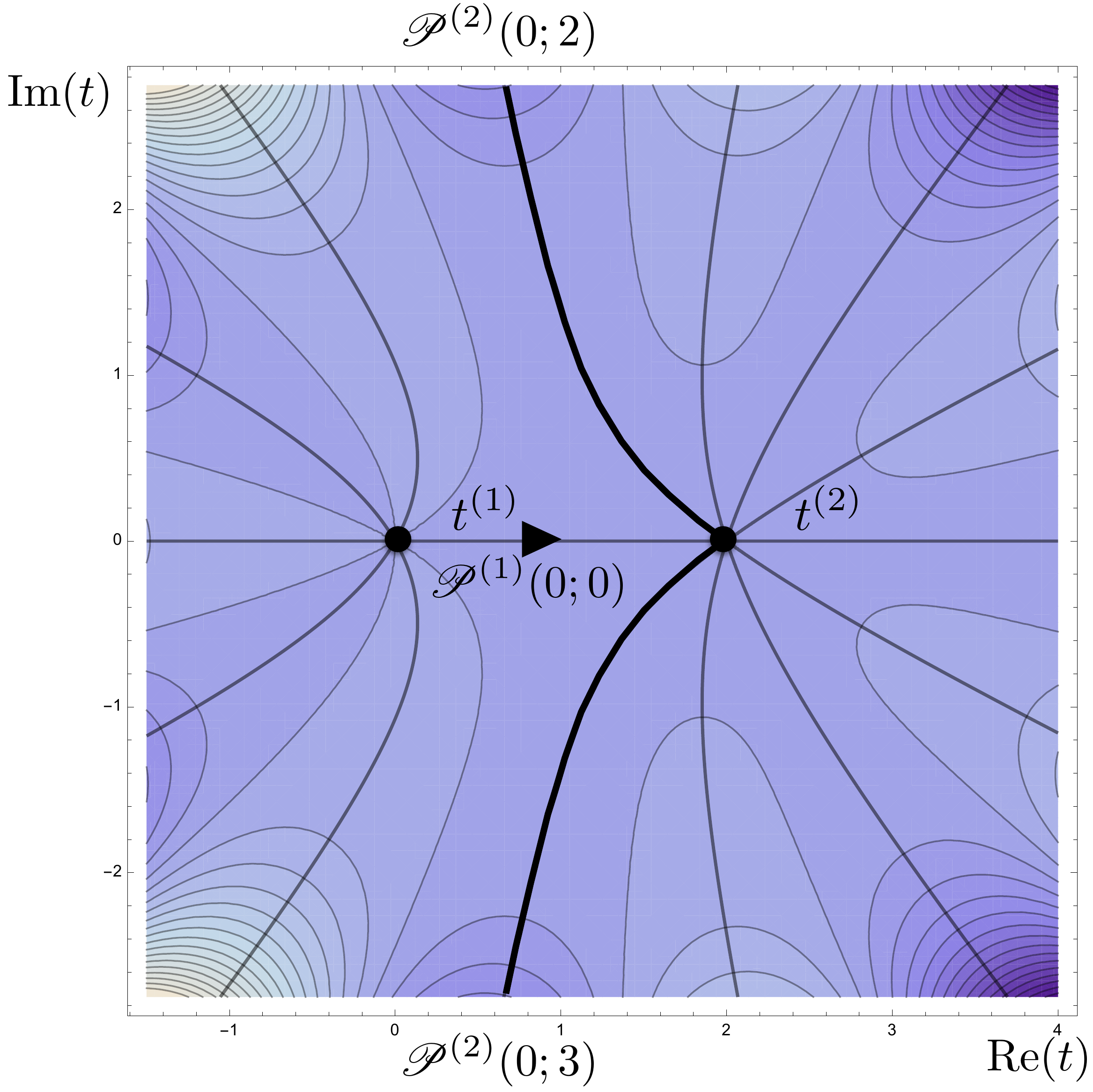}}
\subfigure[\hspace{0 mm} ]{\includegraphics[width=0.32\hsize]{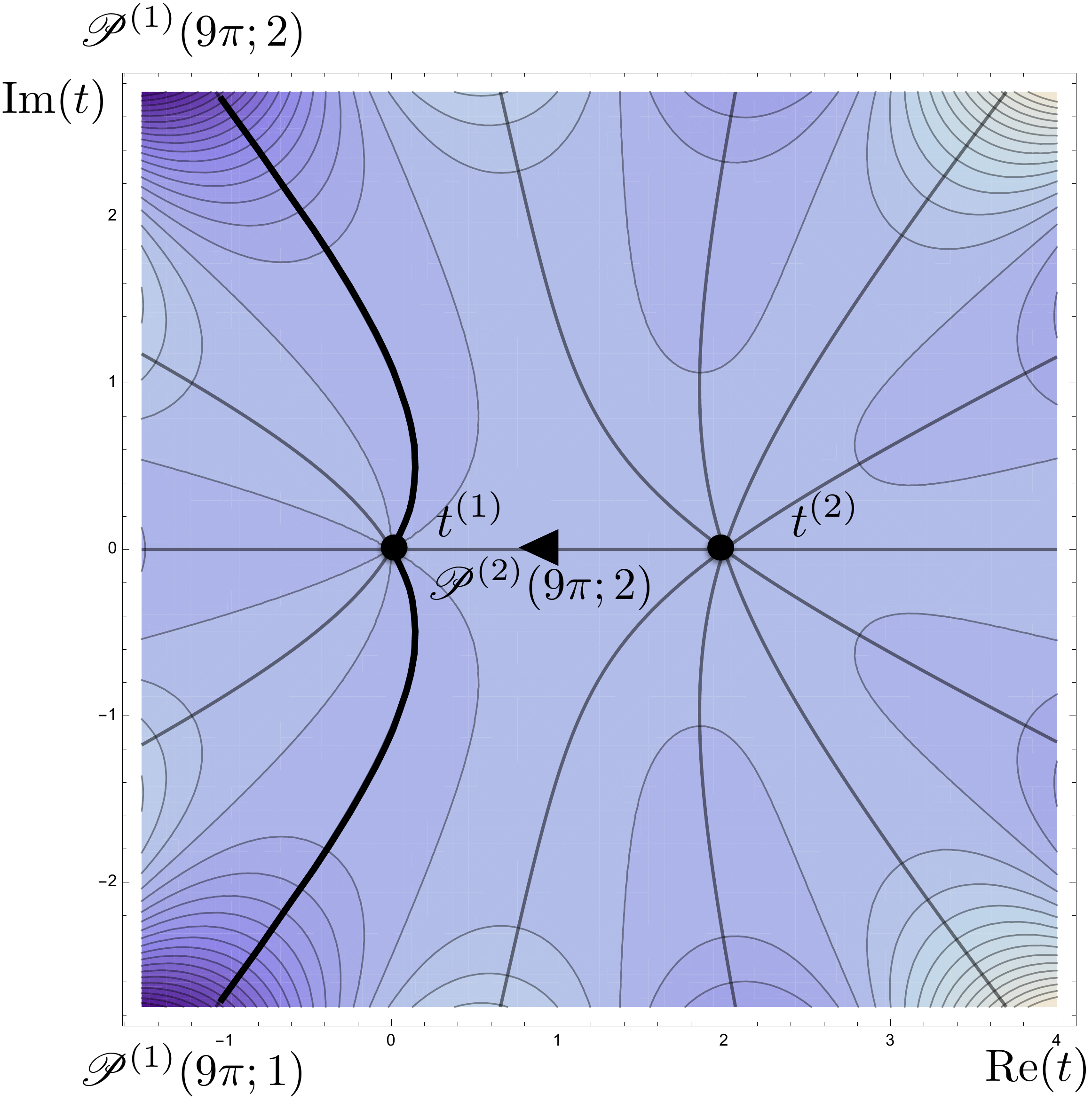}}

\caption{(a) Steepest descent paths in the complex $t$-plane passing through the third order saddle $t^{(1)}$ ($\omega_{1}=3$) and the fifth order saddle $t^{(2)}$ ($\omega_{2}=5$) 
between labelled valleys $V_{j}$, $j=1, 2, \dots, 6$ at infinity for $\theta=-\frac{\pi}{4}$. The path of integration chosen is ${\mathscr P}^{(1)}(-\frac{\pi}{4},0)$ which runs between $t^{(1)}$ and 
$V_{3}$. (b) The rotated steepest descent path ${\mathscr P}^{(1)}(0,0)$, emerging from $t^{(1)}$ connects with $t^{(2)}$ at the Stokes phenomenon $\theta_{12}^{+}=0$. The bold 
lines are the steepest paths that are used in the Level 1 hyperasymptotic expansion about $t^{(1)}$ (\ref{Step3}), (\ref{double}). (c) The steepest descent path ${\mathscr P}^{(2)}(9\pi,
\alpha_{2})$, emerging from $t^{(2)}$ connects with $t^{(1)}$ at the Stokes phenomenon $\theta_{121}^{+}=9\pi$. The bold lines are the steepest paths that are used in the Level 2 hyperasymptotic expansion about $t^{(1)}$  (\ref{Step5}), (\ref{double}). (Or Level 1 hyperasymptotic expansion about $t^{(2)}$.)}

\label{figure8}
\end{figure}

\begin{figure}
\centering\includegraphics[width=0.6\textwidth]{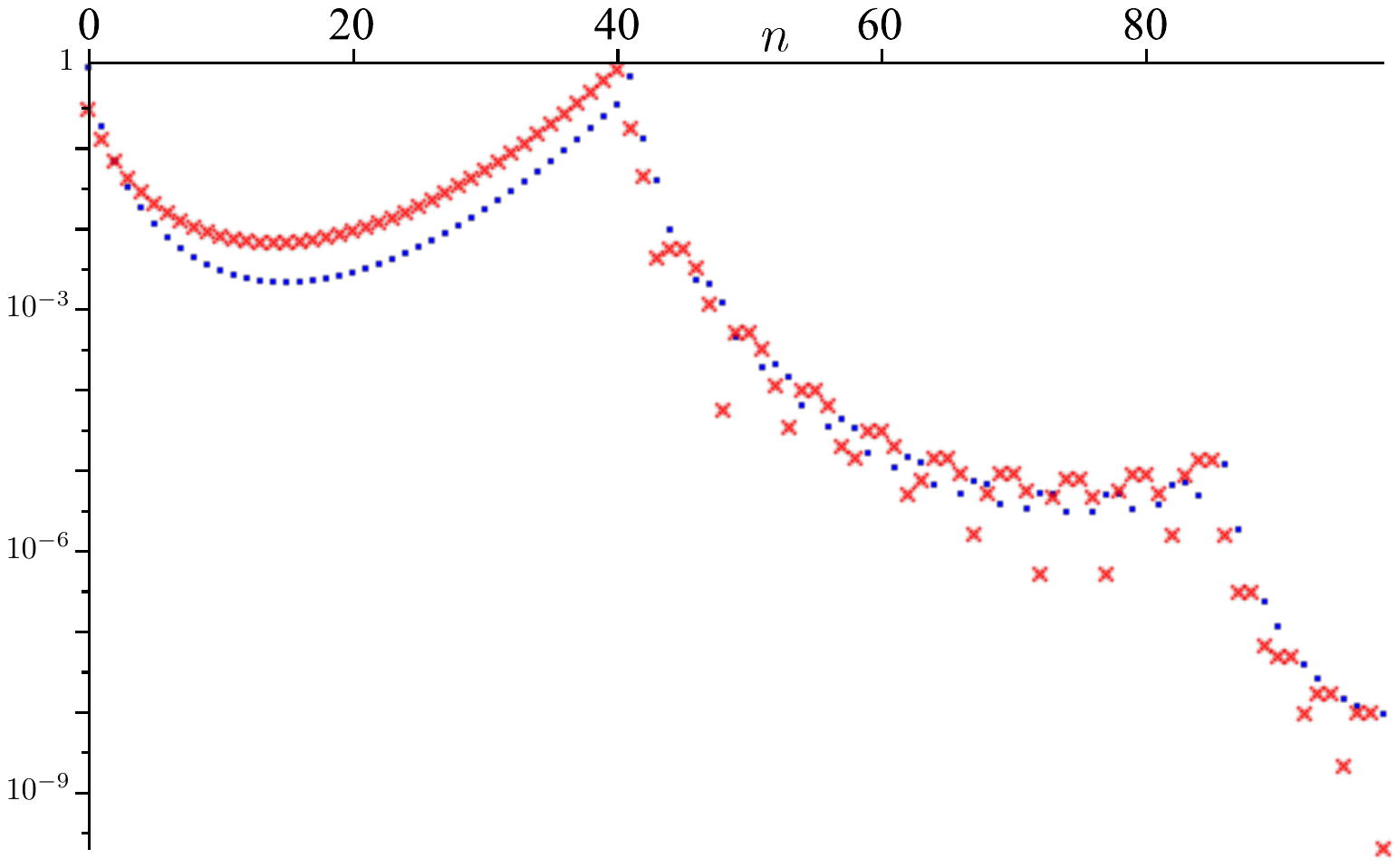}%
\caption{For example 2: The modulus of the $n^{\rm th}$ term in the Level 2 hyperasymptotic expansion (blue dots), and the modulus of the remainder after taking $n$ terms in the approximation (red crosses).}
\label{figure9}
\end{figure}

\section{Example 3: Swallowtail and the adjacency of the saddles}\label{sec8}

In this example we apply hyperasymptotic techniques to determine the relative adjacency, and hence which saddles would contribute to the 
exact remainder terms of an expansion, using algebraic, rather than geometric means.  We choose to illustrate this using the swallowtail integral (\cite[\href{http://dlmf.nist.gov/36.2}{\S36.2}]{NIST:DLMF}).

For the swallowtail integral the bifurcation set is given in \cite[\href{http://dlmf.nist.gov/36.4.E7}{eq. \S36.4.7}]{NIST:DLMF}
and with the notation in this reference we take $t=\frac12\i-\frac14$ and $z=\frac56\i-\frac{25}8$.   (The choice of 
complex parameters is to force one of the saddles to be non-adjacent, see below.)

The resulting semi-infinite contour integral that we will study is again integral \eqref{eq1}, but now with $g(t) \equiv 1$ and
$$f(t)=t^5+\ifrac5{24}\left(4\i-15\right)t^3+\ifrac{45}{16}\left(2\i-1\right)t^2+\ifrac5{256}\left(101+168\i\right)t.$$
The saddle points are $t^{(1)}=\frac74-\frac12\i$, $t^{(2)}=-\frac54-\frac12\i$, and $t^{(3)}=\frac12\i-\frac14$, with  $\omega_1=\omega_2=2$ 
and $\omega_3=3$.
Once again, the polynomial form of $f(t)$ means that we may take $z=\e^{\i\theta}$ with the choice of $\theta=-\frac{\pi}{4}$. To obtain the Level 1 
hyperasymptotic approximation we find that
$$|\mathcal{F}_{12}^{+}|=\frac{9\sqrt{109}}{4},\quad |\mathcal{F}_{13}^{+}|=\frac{125\sqrt{5}}{12},\quad
\theta_{12}^+=3\pi-\arctan\ifrac{10}{3},\quad \theta_{13}^+=3\pi-\arctan\ifrac{278}{29}.$$
It follows that $\alpha_{12}^+=1$ and $\alpha_{13}^+=0$.
We write the Level 1 hyperasymptotic approximation as
\begin{gather}\begin{split}\label{SwT1}
&T^{(1)} (z ;0)=\sum\limits_{r = 0}^{N - 1} \frac{T_r^{(n)}(0)}{z^{r/2}} +K_{12}\frac{z^{(1-N)/2}}{4\pi\i} \sum_{r=0}^{N_1^{(2)}-1}{\bf T}_r^{(2)}(1)
\HypTermOne{z}{\frac{N+1}{2}-\frac{r+1}{2}}{2}{|\mathcal{F}_{12}^+|\e^{\i(\pi-\theta_{12}^{+})}}\\
&\qquad+K_{13}\frac{z^{(1-N)/2}}{6\pi\i} \sum_{r=0}^{N_1^{(3)}-1}{\bf T}_r^{(3)}(0)
\HypTermOne{z}{\frac{N+1}{2}-\frac{r+1}{3}}{2}{|\mathcal{F}_{13}^+|\e^{\i(\pi-\theta_{13}^{+})}}+R_{1}^{(1)}(z;0).
\end{split}\end{gather}
Note that we have here introduced unknown constant prefactors $K_{nm}$ into the expression for the Level 1 hyperasymptotic expansion 
(\ref{Step3}).  Each constant will be equal to 1 if the saddles $t^{(n)}$ and $t^{(m)}$ are adjacent, and zero otherwise. We could determine 
these constants by examining how the steepest descent contours deform as $\theta$ is varied.   However, here we illustrate their 
algebraic calculation.  
These constants appear in the late term expansion \eqref{LateTerms} (which also follows
from \eqref{SwT1}) as follows:
\begin{gather}\begin{split}\nonumber 
T_N^{(1)}(0)=&
\frac{K_{12}}{4\pi\i} \sum_{r=0}^{N_1^{(2)}-1}{\bf T}_r^{(2)}(1)
\frac{\e^{\i\theta_{12}^+\left(\frac{N+1}{2}-\frac{r+1}{2}\right)}\Gamma\left(\frac{N+1}{2}-\frac{r+1}{2}\right)}
{|\mathcal{F}_{12}^+|^{\frac{N+1}{2}-\frac{r+1}{2}}}\\
&+\frac{K_{13}}{6\pi\i} \sum_{r=0}^{N_1^{(3)}-1}{\bf T}_r^{(3)}(0)
\frac{\e^{\i\theta_{13}^+\left(\frac{N+1}{2}-\frac{r+1}{3}\right)}\Gamma\left(\frac{N+1}{2}-\frac{r+1}{3}\right)}
{|\mathcal{F}_{13}^+|^{\frac{N+1}{2}-\frac{r+1}{3}}}+\widetilde{R}_{1}^{(1)}(N;0).
\end{split}\end{gather}
In this (asymptotic) expression, everything is known except, $K_{12}$ and $K_{13}$. Hence if we take two high orders $N=50$ and $N=51$ 
and set $\widetilde{R}_{1}^{(1)}(N;0)=0$ we obtain 2 linear algebraic equations with 2 unknowns. 
The optimal number of terms on the right-hand side may be calculated from \eqref{optimal} and are $N_1^{(2)}=7$ and $N_1^{(3)}=11$. 
Hence we can solve this simultaneous set of equations to obtain numerical approximations for $K_{12}$ and $K_{13}$ as
$$K_{12}=-0.00123+0.00095\i,\qquad K_{13}=1.00076+0.00060\i.$$
Given that the $K_{nm}$ are quantised as integers, within the limits of the errors at this stage, we may infer that $K_{12}=0$ and $K_{13}=1$.  

Hence we may assert that $t^{(3)}$ is adjacent to $t^{(1)}$, but $t^{(2)}$ is not.  This may be confirmed geometrically by consideration of the steepest paths.

\section{Discussion}\label{sec9}

The main results of his paper are the exact remainder terms \eqref{Step1}, \eqref{Step2}, the hyperasymptotic re-expansions \eqref{Step3}, 
\eqref{Step5}, \eqref{Step6}, with novel hyperterminants \eqref{defhyperterminant}, the asymptotic form for the late 
coefficients \eqref{LateTerms} and the improved error bounds for the remainder of an asymptotic 
expansion involving saddle points \eqref{bound}, degenerate or otherwise.  We have illustrated the application of these results to the better-than-exponential asymptotic expansions and calculations of integrals with semi-infinite contours and degenerate saddles.

The results of this paper are more widely applicable, for example to broadening the class of differential equations for which a 
hyperasymptotic expansion may be derived using a Borel transform approach.  We observe 
that all the examples in this paper are of the form
\begin{equation}\nonumber
w(z)=\int_{t^{(1)}}^\infty \e^{-zf(t)}g(t)\d t,
\end{equation}
in which $f(t)$ and $g(t)$ are polynomials in $t$. (In fact $g(t)\equiv 1$.) Using computer algebra, it is not difficult to construct the 
corresponding inhomogeneous linear ordinary differential equations for $w(z)$:
\begin{equation}\label{ODE}
\sum_{p=0}^P a_p(z)w^{(p)}(z)=h(z),
\end{equation}
in which the $a_p(z)$'s and $h(z)$ are polynomials.   

For our second example with $(\omega_{1}, \omega_{2})=(3,5)$ we find $P=6$, the $a_p(z)$'s are polynomials of order 9, and $h(z)$ is of 
order 6. Integrals involving combinations pairs of the contours ${\mathscr P^{(n)}}$ are solutions of the homogeneous version of \eqref{ODE}.

In that example, for the first saddle point we have $\omega_1=3$, and hence, there are 2 independent double infinite integrals through this saddle, and for the second saddle point
we have $\omega_2=5$, and hence, there are 4 independent double infinite integrals through the second saddle. Thus, $P=2+4$.

The differential equation \eqref{ODE} has an irregular singularity of rank one at infinity, but we are dealing with the exceptional cases.  That 
is, the solutions all have initial terms proportional to $\exp(\lambda_p z)z^{\mu_p}$ but 
now with coinciding $\lambda_p$'s. For example, in our second example we have two 
distinct solutions with $\lambda_1=\lambda_2=0$ and four other different solutions but 
each with $\lambda_3=\lambda_4=\lambda_5=\lambda_6=\frac{32}7$.

Note also, that $h(z)$ in \eqref{ODE} is a polynomial in $z$. Hence we should expect a particular integral of \eqref{ODE} to involve only 
integer powers of $z$. However, the particular integral $w(z)=z^{-1/3}T^{(1)} (z;0)$ has, according to \eqref{eq35b},
an asymptotic expansion in inverse powers of $z^{1/3}$.  The resolution of this paradox is that the combination of such solutions  
$$w(z)=\frac{z^{-1/3}}{3}\left(T^{(1)} (z;0)+T^{(1)} (z;1)+T^{(1)} (z;2)\right)$$ 
is itself a particular integral, but contains only integer powers.  This solution involves a star-shaped contour of integration, typically not studied if the problem is posed in terms of integrals alone.

We also remark that differential equations of the form \eqref{ODE} will give us recurrence relations for the coefficients in the asymptotic expansions, and these are, of course, much more efficient than our formula \eqref{eq11}.

\appendix

\section{Computation of the generalised hyperterminants}\label{appA}

In this appendix we relate the generalised hyperterminants \eqref{defhyperterminant} to the simpler ones given in \cite{OD98c} and thereby develop an efficient method to calculate them.  

First, the following theorem improves on the main theorem in \cite{OD98c}.

\begin{theorem} For $k\geq 0$, $|\arg z+\arg \sigma_0|<\pi$ and $0<\arg \sigma_{j}-\arg \sigma_{j-1}<2\pi$, $j\geq 1$, $\mathop{\rm Re}(M_1)>2$ and $\mathop{\rm Re}(M_j)>1$, $j\neq 1$, we have the convergent expansion
\[
F^{(k + 1)} \left( z;\begin{array}{c}
   {M_0 ,}  \\
   {\sigma _0 ,}  \\
\end{array}\!\!\!\begin{array}{c}
   { \ldots ,}  \\
   { \ldots ,}  \\
\end{array}\!\!\!\begin{array}{c}
   {M_k }  \\
   {\sigma _k }  \\
\end{array} \right) = \sum\limits_{n = 0}^\infty
A^{(k + 1)} \left( n;\begin{array}{c}
   {M_0 ,}  \\
   {\sigma _0 ,}  \\
\end{array}\!\!\!\begin{array}{c}
   { \ldots ,}  \\
   { \ldots ,}  \\
\end{array}\!\!\!\begin{array}{c}
   {M_k }  \\
   {\sigma _k }  \\
\end{array} \right)U(n + 1,2 - M_0 ,z\sigma _0 ) ,
\]
where
\[
A^{(1)} \left( {n;\begin{array}{c}
   {M_0}  \\
   {\sigma _0}  \\
\end{array}} \right) = \delta_{n,0} \e^{M_0 \pi \i} \sigma_0^{1-M_0} \Gamma(M_0),
\]
\begin{align*}
A^{(2)} \left( n;\begin{array}{c}
   {M_0 ,}  \\
   {\sigma _0 ,}  \\
\end{array}\begin{array}{c}
   {M_1 }  \\
   {\sigma _1 }  \\
\end{array} \right) = & - \e^{\pi M_0 \i} \sigma _0^{2 - M_0  - M_1 } \left( {\e^{ - \pi \i} \frac{{\sigma _1 }}{{\sigma _0 }}} \right)^{n - M_1  + 1}  \Gamma (M_0  + n)\Gamma (M_1 ) \\ 
& \times \frac{{n!\Gamma (M_0  + M_1  - 1)}}{{\Gamma (M_0  + M_1  + n)}}{}_2F_1 \left( {M_0  + n,n + 1\atop M_0  + M_1  + n};1 + \frac{{\sigma _1 }}{{\sigma _0 }} \right) ,
\end{align*}
and when $k\geq 1$,
\begin{align*}
& A^{(k + 1)} \left( n;\begin{array}{c}
   {M_0 ,}  \\
   {\sigma _0 ,}  \\
\end{array}\!\!\!\begin{array}{c}
   { \ldots ,}  \\
   { \ldots ,}  \\
\end{array}\!\!\!\begin{array}{c}
   {M_k }  \\
   {\sigma _k }  \\
\end{array} \right) = \;  \e^{\pi M_0 \i} \sigma _0^{1 - M_0 } \left( {\e^{ - \pi \i} \frac{{\sigma _1 }}{{\sigma _0 }}} \right)^n \Gamma (M_0  + n)\Gamma (M_0  + M_1  - 1)\\
&\qquad \times \sum\limits_{m = 0}^\infty  \frac{(n + m)!\,\,
A^{(k)} \left( m;\begin{array}{c}
   {M_1 ,}  \\
   {\sigma _1 ,}  \\
\end{array}\!\!\!\begin{array}{c}
   { \ldots ,}  \\
   { \ldots ,}  \\
\end{array}\!\!\!\begin{array}{c}
   {M_k }  \\
   {\sigma _k }  \\
\end{array} \right)
}{m!\Gamma (M_0  + M_1  + n + m)} 
{}_2F_1 \left( {M_0  + n,n + m + 1\atop M_0  + M_1  + n + m};1 + \frac{\sigma _1}{\sigma _0} \right).
\end{align*}
Here ${}_2F_1$ stands for the hypergeometric function \cite[\href{http://dlmf.nist.gov/15.2}{\S15.2}]{NIST:DLMF}.
\end{theorem}

The proof of this theorem  is very similar to the one for Theorem 2 in  \cite{OD98c}. The main difference here is that we must be more careful with the definitions of the phases and use the restrictions $0<\arg \sigma_{j}-\arg \sigma_{j-1}<2\pi$.   This removes any phase-related ambiguity in the calculation of the hyperterminants.  

Note that the representation of $A^{(k+1)}$ is a convergent infinite series. In a practical implementation it is necessary to truncate the series appropriately.
In our numerical examples we took 40 terms of the convergent series and checked, by taking successively more terms, that this truncation gave us sufficient correct digits in the evaluation of the corresponding hyperterminants.

In the theorem above we also require, for example, that $0<\arg \sigma_{1}-\arg \sigma_{0}<2\pi$.
In fact, one often encounters the case $\sigma_0=\sigma_1$ and some care is then needed to evaluate the ${}_2F_1$.
Numerical methods to evaluate the (confluent) hypergeometric function may be found in \cite{Gil07}.

With these phase clarifications, the generalised hyperterminants \eqref{defhyperterminant} can be expressed in terms of the ones above as 
follows. 

First, by rationalisation, we have
\begin{align*}
& \HypTermK{k+1}{z}{%
\begin{array}{c}
   {M_0 ,}  \\
   {\omega _0 ,}  \\
   {\sigma _0  ,}  \\
\end{array}\begin{array}{c}
   { \ldots ,}  \\
   { \ldots ,}  \\
   { \ldots ,}  \\
\end{array}\begin{array}{c}
   {M_k }  \\
   {\omega _k }  \\
   {\sigma _k }  \\
\end{array}} \\ 
&\qquad = \sum\limits_{\ell _0  = 0}^{\omega _0  - 1} z^{1 - (\ell _0  + 1)/\omega _0 } \int_0^{\left[ {\pi  - \arg \sigma _0 } \right]}  \cdots \int_0^{\left[ {\pi  - \arg \sigma _k } \right]} 
\prod\limits_{j = 1}^k \frac{{\e^{\sigma _0 t_0 } t_0^{M_0  + \ell _0 /\omega _0  - 1} }}{{z - t_0 }} \\ 
& \qquad\qquad\times \sum\limits_{\ell _j  = 0}^{\omega _j  - 1} {\frac{{\e^{\sigma _j t_j } t_{j - 1}^{1 - (\ell _j  + 1)/\omega _j } t_j^{M_j  + \ell _j /\omega _j  - 1} }}{{t_{j - 1}  - t_j }}} \d t_k  \cdots \d t_0    \\
&\qquad = \sum\limits_{\ell _0  = 0}^{\omega _0  - 1}  \cdots \sum\limits_{\ell _k  = 0}^{\omega _k  - 1} z^{1 - (\ell _0  + 1)/\omega _0 } \int_0^{\left[ {\pi  - \arg \sigma _0 } \right]}  \cdots 
\int_0^{\left[ {\pi  - \arg \sigma _k } \right]} \frac{{\e^{\sigma _0 t_0 } t_0^{M_0  + \ell _0 /\omega _0  - (\ell _1  + 1)/\omega _1 } }}{{z - t_0 }} \\
&\qquad\qquad \times \left( \prod\limits_{j = 1}^{k - 1} {\frac{{\e^{\sigma _j t_j } t_j^{M_j  + \ell _j /\omega _j  - (\ell _{j + 1}  + 1)/\omega _{j + 1} } }}{{t_{j - 1}  - t_j }}}  \right)
\frac{{\e^{\sigma _k t_k } t_k^{M_k  + \ell _k /\omega _k  - 1} }}{{t_{k - 1}  - t_k }}\d t_k  \cdots \d t_0  .  
\end{align*}
We make the changes of integration variables from $t_0$ to $s_0$ and from $t_j$ to $s_j$ ($1\leq j \leq k$) via $
t_0  = s_0 \e^{2\pi \gamma _0 \i}$ and $t_j  = s_j \e^{2\pi (\gamma _{j-1}  + \gamma _j )\i}$. Here, the integers $\gamma_0$ and $\gamma_j$ are chosen so that
$\left| {\arg z + \arg \sigma _0  + 2\pi \gamma _0 } \right| < \pi $ and $0 < \arg \sigma _j -\arg \sigma _{j-1} + 2\pi \gamma _j  < 2\pi$. 

Thus, we can finally relate the ${\bf F}^{(k+1)}$ to the $F^{(k+1)}$ with the result:
\begin{align*}
& \HypTermK{k+1}{z}{%
\begin{array}{c}
   {M_0 ,}  \\
   {\omega _0 ,}  \\
   {\sigma _0  ,}  \\
\end{array}\begin{array}{c}
   { \ldots ,}  \\
   { \ldots ,}  \\
   { \ldots ,}  \\
\end{array}\begin{array}{c}
   {M_k }  \\
   {\omega _k }  \\
   {\sigma _k }  \\
\end{array}} \\ 
& = \sum\limits_{\ell _0  = 0}^{\omega _0  - 1}  \cdots \sum\limits_{\ell _k  = 0}^{\omega _k  - 1} z^{1 - (\ell _0  + 1)/\omega _0 } 
\e^{2\pi\i \left(\gamma _{k - 1} (M_{k - 1}  + M_k  + \frac{\ell _{k - 1}}{\omega _{k - 1}}  - \frac{1}{\omega _k})+\gamma _k (M_k  + \frac{\ell _k}{\omega _k} )\right)} \\ 
& \qquad\times \prod\limits_{j = 0}^{k - 2} {\e^{2\pi\i \gamma _j (M_j  + M_{j + 1}  + \frac{\ell _j }{\omega _j}  - \frac{1}{\omega _{j + 1}}  - \frac{\ell _{j + 2}  + 1}{\omega _{j + 2}} )} } \\ 
& \qquad\times \int_0^{\left[ {\pi  - \arg \sigma _0  - 2\pi \gamma _0 } \right]}  \cdots \int_0^{\left[ {\pi  - \arg \sigma _k  - 2\pi (\gamma _{k - 1}  + \gamma _k )} \right]} 
\frac{{\e^{\sigma _0s_0+\cdots+\sigma_k s_k } s_0^{M_0  + \ell _0 /\omega _0  - (\ell _1  + 1)/\omega _1 } }}{{z - s_0 }} \\ 
& \qquad\times \left( \prod\limits_{j = 1}^{k - 1} {\frac{{ s_j^{M_j  + \ell _j /\omega _j  - (\ell _{j + 1}  + 1)/\omega _{j + 1} } }}{{s_{j - 1}  - s_j }}}  \right)
\frac{{ s_k^{M_k  + \ell _k /\omega _k  - 1} }}{{s_{k - 1}  - s_k }}\d s_k  \cdots \d s_0    \\
& = \sum\limits_{\ell _0  = 0}^{\omega _0  - 1}  \cdots \sum\limits_{\ell _k  = 0}^{\omega _k  - 1} z^{1 - (\ell _0  + 1)/\omega _0 } 
\e^{2\pi\i \left(\gamma _{k - 1} (M_{k - 1}  + M_k  + \frac{\ell _{k - 1}}{\omega _{k - 1}}  - \frac{1}{\omega _k})+\gamma _k (M_k  + \frac{\ell _k}{\omega _k} )\right)} \\ 
& \qquad\times \prod\limits_{j = 0}^{k - 2} {\e^{2\pi\i \gamma _j (M_j  + M_{j + 1}  + \frac{\ell _j }{\omega _j}  - \frac{1}{\omega _{j + 1}}  - \frac{\ell _{j + 2}  + 1}{\omega _{j + 2}} )} } \\  
& \qquad\times F^{(k + 1)} \left( z;
   {M_0  + \frac{\ell _0}{\omega _0}  - \frac{\ell _1  + 1}{\omega_1} +1,\atop \sigma _0 \e^{2\pi \gamma _0 \i},}
   {M_1  + \frac{\ell _1}{\omega _1}  -\frac{\ell _2  + 1}{\omega _2} +1,\atop   \sigma _1 \e^{2\pi (\gamma _0  + \gamma _1 )\i} ,}
   ~{\ldots,\atop\ldots,}~ {M_k  + \frac{\ell _k}{\omega _k}\atop \sigma _k \e^{2\pi (\gamma _{k - 1}  + \gamma _k )\i}}
\right) .
\end{align*}

\section{Bounds for the generalised first-level hyperterminant}\label{appB}

\begin{proposition}\label{hyperbound} For any positive real $M$ and positive integer $\omega$, we have
\[
\left| {\frac{z^{1/\omega}}{{\Gamma (M  )}}
\HypTermOne{z}{M}{\omega}{1}}\right| \le \begin{cases} 1 & \!\!\!\text{ if } \;  |\theta| \leq \frac{\pi}{2}\omega , \\ 
\min\left(\left| {\csc \left( {\frac{{\theta}}{{\omega  }}} \right)} \right|,\omega\sqrt {\e\left( {M + \frac{1}{2}} \right)} \right) & \!\!\!\text{ if } \; \frac{\pi}{2}\omega  < |\theta| \leq \pi \omega , \\ 
\frac{{\omega\sqrt {2\pi M} }}{{\left| {\cos \theta} \right|^M }} + \omega\sqrt {\e\left( {M + \frac{1}{2}} \right)} & \!\!\!\text{ if } \; \pi \omega  < |\theta| < \pi \omega+\frac{\pi}{2}. \end{cases}
\]
If $\omega=1$, the quantity $\sqrt {\e\left( {M + \frac{1}{2}} \right)}$ can be replaced by
\begin{equation}\label{gammaratio}
\sqrt \pi  \frac{\Gamma \big( \frac{M}{2} + 1 \big)}{\Gamma \big( \frac{M}{2} + \frac{1}{2} \big)} + 1,
\end{equation}
which is asymptotic to $\sqrt {\frac{\pi }{2}\left( M + \frac{1}{2} \right)} $ as $M\to \infty$ and hence yields a sharper bound for large $M$.
\end{proposition}

\begin{proof} The case $\omega=1$ was proved in a recent paper by Nemes \cite[Propositions B.1 and B.3]{Nemes17}. 
For the general case, let $M$ be any positive real number and $\omega$ be any positive integer. The integral representation of the first generalised hyperterminant can be re-written
\begin{equation}\label{hypint}
\frac{{z^{1/\omega  } }}{{\Gamma (M  )}}
\HypTermOne{z}{M}{\omega}{1}
 = \frac{\e^{\pi M \i}}{\Gamma (M  )}\int_0 ^{\infty } \frac{\e^{ - t} t^{M   - 1} }{1 + (t/z  )^{1/\omega  } }\d t ,
\end{equation}
provided that $|\theta|<\pi \omega$. For $t \geq 0$, we have
\begin{equation}\label{ineq}
\left| 1 + \frac{t}{w} \right| \ge \begin{cases} 1 & \text{ if } \;  |\arg w| \leq \frac{\pi}{2}, \\ |\sin \left(\arg w\right)| & \text{ if } \; \frac{\pi}{2} < |\arg w| < \pi,\end{cases}
\end{equation}
and therefore
\begin{align*}
\left| {\frac{{z^{1/\omega  } }}{{\Gamma (M  )}}\HypTermOne{z}{M}{\omega}{1}} \right| 
& \le \frac{1}{\Gamma (M  )}\int_0 ^{ \infty } {\frac{{\e^{ - t} t^{M  - 1} }}{{\left| {1 + (t/z   )^{1/\omega  } } \right|}}\d t}  \\ & \le  \begin{cases} 1 & \text{ if } \;  |\theta| \leq \frac{\pi}{2}\omega , \\ \left| {\csc \left( {\frac{{\theta}}{{\omega  }}} \right)} \right| & \text{ if } \; \frac{\pi}{2}\omega  < |\theta| < \pi \omega .\end{cases}
\end{align*}
We continue by showing that the absolute value of the left-hand side of \eqref{hypint} is bounded by $\omega \sqrt {\e\left( {M + \frac{1}{2}} \right)}$ when $\frac{\pi }{2}\omega < \theta \le \pi \omega$. (The analogous bound for the range $-\pi \omega \leq \theta < -\frac{\pi }{2}\omega$ follows by taking complex conjugates.) For this purpose, we deform the contour of integration in \eqref{hypint} by rotating it through an acute angle $\varphi$. Thus, by appealing to Cauchy's theorem and analytic continuation, we have, for arbitrary $0<\varphi<\frac{\pi}{2}$, that
\[
\frac{{z^{1/\omega  }}}{{\Gamma (M  )}}\HypTermOne{z}{M}{\omega}{1} = 
\frac{\e^{\pi M \i}}{{\Gamma (M  )}}\left( {\frac{{\e^{\i\varphi } }}{{\cos \varphi }}} \right)^{M  } 
\int_0 ^{ \infty } {\frac{{\e^{ - \frac{{\e^{\i\varphi } u}}{{\cos \varphi }}} u^{M   - 1} }}{{1 + \left( {\frac{{\e^{\i\varphi } u}}{{z  \cos \varphi }}} \right)^{1/\omega  } }}\d u} 
\]
when $\frac{\pi }{2}\omega < \theta \le \pi \omega$. Employing the inequality \eqref{ineq}, we find that
\begin{align*}
\left| {\frac{{z^{1/\omega  } }}{{\Gamma (M  )}}\HypTermOne{z}{M}{\omega}{1}} \right| 
& \le \frac{1}{\Gamma (M  )}\frac{1}{\cos^{M  } \varphi}\int_0 ^{ \infty } {\frac{{\e^{ - u} u^{M - 1} }}{{\left| {1 + \left( {\frac{{\e^{\i\varphi } u}}{{z  \cos \varphi }}} \right)^{1/\omega  } } \right|}}\d u} 
\\ & \le \frac{1}{\cos^{M  } \varphi}
 \times \begin{cases} 1 & \text{ if } \; \frac{\pi}{2}\omega < \theta \leq \frac{\pi}{2}\omega +\varphi, \\ \left| {\csc \left( {\frac{\theta-\varphi}{{\omega  }}} \right)} \right| & \text{ if } \; \frac{\pi}{2}\omega +\varphi < \theta \le \pi \omega .\end{cases}
\end{align*}
We now choose the value of $\varphi$ which approximately minimizes the right-hand side of this inequality when $\theta =\pi \omega$, namely $\varphi  = \arctan (M^{ - 1/2} )$. We may then claim that
\[
\frac{1}{{\cos ^M (\arctan (M^{ - 1/2} ))}} = \left( {1 + \frac{1}{M}} \right)^{M/2}  \le \omega\sqrt {\e\left( {M + \frac{1}{2}} \right)} ,
\]
when $\frac{\pi }{2}\omega < \theta \le \frac{\pi }{2}\omega + \arctan (M^{ - 1/2} )$, where the last inequality can be obtained by means of elementary analysis. 
In the remaining case $\frac{\pi }{2}\omega + \arctan (M^{ - 1/2} ) < \theta \le \pi \omega$, we have
\begin{align*}
& \frac{{\left| {\csc \left( {\frac{{\theta - \arctan (M^{ - 1/2} )}}{\omega}} \right)} \right|}}{{\cos ^M (\arctan (M^{ - 1/2} ))}} \le 
\frac{{\left| {\csc \left( \pi  - \frac{\arctan (M^{ - 1/2} )}{\omega} \right)} \right|}}{{\cos ^M (\arctan (M^{ - 1/2} ))}} \\
&\qquad= \left( 1 + \frac{1}{M} \right)^{M/2} \csc \left( {\frac{{\arctan (M^{ - 1/2} )}}{\omega}} \right)
 \le \left( 1 + \frac{1}{M} \right)^{M/2} \omega \csc ( \arctan (M^{ - 1/2} )) \\
 &\qquad= \omega\left( 1 + \frac{1}{M} \right)^{(M + 1)/2} \sqrt M \le \omega\sqrt {\e\left( {M + \frac{1}{2}} \right)} .
\end{align*}
Here we have used the convexity of $\csc (x)$ for $0<x<\frac{\pi}{2}$, and that the quantity $\left( {1 + \frac{1}{M}} \right)^{(M+1)/2} \sqrt {\frac{M}{M + a}}$, as a function of $M>0$, increases monotonically if and only if $a\geq \frac{1}{2}$, in which case it has limit $\sqrt{\e}$.

We finish by proving the claimed bound for the range $\pi \omega  < |\theta| < \pi \omega+\frac{\pi}{2}$.  It is sufficient to consider the range $\pi \omega  < \theta < \pi \omega+\frac{\pi}{2}$, as the estimates for $-\pi \omega-\frac{\pi}{2} < \theta < -\pi \omega$ can be derived by taking complex conjugates. The proof is based on the functional relation
\[
\frac{{z^{1/\omega}  }}{{\Gamma (M)}}\HypTermOne{z}{M}{\omega}{1} = \frac{{2\pi \i \omega\left(z \e^{ - \pi \i(\omega - 1)} \right))^M}}{{\Gamma (M)\e^{ z \e^{ - \pi \i \omega} }}}  
+ \frac{{(z\e^{ - 2\pi \i \omega} )^{1/\omega} }}{{\Gamma (M)}}\HypTermOne{z\e^{ - 2\pi \i \omega}}{M}{\omega}{1}
\]
(see \cite[eq. (A.13)]{Paris14}). From this functional relation, we can infer that
\begin{align*}
\left| {\frac{{z^{1/\omega}}}{{\Gamma (M)}}\HypTermOne{z}{M}{\omega}{1}} \right| & \le \frac{{2\pi \omega\left| z \right|^M}}{\Gamma (M)\e^{ \left| {z} \right|\left| {\cos \theta} \right|} }
+ \left| {\frac{{(z\e^{ - 2\pi \i \omega} )^{1/\omega} }}{{\Gamma (M)}}\HypTermOne{z\e^{ - 2\pi \i \omega}}{M}{\omega}{1}} \right| \\ 
& \le \frac{{2\pi \omega\left| z \right|^M}}{\Gamma (M)\e^{ \left| {z} \right|\left| {\cos \theta} \right|} }  + \omega\sqrt {\e\left( {M + \frac{1}{2}} \right)} .
\end{align*}
Notice that the quantity $r^M \e^{ - r a }$, as a function of $r>0$, takes its maximum value at $r=M/a$ when $a>0$ and $M>0$. We therefore find that
\begin{align*}
\left| {\frac{{z^{1/\omega} }}{{\Gamma (M)}}\HypTermOne{z}{M}{\omega}{1}} \right| & \le \frac{{\omega\sqrt {2\pi M} }}{{\left| {\cos \theta} \right|^M }}\frac{{M^{M - 1/2} \e^{ - M} \sqrt {2\pi } }}{\Gamma (M)} + \omega\sqrt {\e\left( {M + \frac{1}{2}} \right)} 
\\ & \le \frac{{\omega\sqrt {2\pi M} }}{{\left| {\cos \theta} \right|^M }} + \omega\sqrt {\e\left( {M + \frac{1}{2}} \right)} .
\end{align*}
The second inequality can be obtained from the inequality $M^{M - 1/2} \e^{ - M}  \sqrt {2\pi } \le \Gamma \left( M \right)$ for any $M>0$ (see, for instance, \cite[\href{http://dlmf.nist.gov/5.6.E1}{eq.~5.6.1}]{NIST:DLMF}).
\end{proof}

\section{The boundary of the domain $\Delta^{(n)}$}\label{appC}

In this appendix, we prove that the boundary of $\Delta^{(n)}$ can be written as a union $\bigcup\nolimits_m \mathscr{P}^{(m)} (\theta _{nm}^ +  ,\alpha _{nm}^ +  ) \cup -\mathscr{P}^{(m)} (\theta _{nm}^ -  ,\alpha _{nm}^ -  )$, where $\mathscr{P}^{(m)} (\theta _{nm}^ \pm  ,\alpha _{nm}^ \pm  )$ are steepest descent paths emerging from the adjacent saddle $t^{(m)}$ (see Figure \ref{fig2}(b)). For $\alpha _{nm}^ \pm$, see \eqref{alphanm}.

First, we show that as we change $\theta$, the steepest descent path $\mathscr{P}^{(n)}(\theta;\alpha_n)$ varies smoothly, unless, perhaps, it encounters an adjacent saddle point $t^{(m)}$. To see this, consider the map $s(t)$ between the $t$-plane and the $s$-surface, defined by
\[
s = f(t) - f_n.
\]
The steepest descent path $\mathscr{P}^{(n)} (\theta;\alpha_n)$ is mapped into a half-line with phase $2\pi\alpha_n-\theta$ emerging from the origin as an $\omega_n^{\rm th}$-order branch point on the $s$-surface. As this half-line is rotated on the $s$-surface, the corresponding steepest descent path varies smoothly, unless we encounter a singularity of the inverse map $t(s)$. Since $f(t)$ is holomorphic in the closure of $\Delta^{(n)}$, and $|f(t)| \to  \infty$ as $t\to \infty$ in $\Delta^{(n)}$, the only singularities of $t(s)$ are branch points located at the images of the saddle points of $f(t)$ under the map $s(t)$. When the half line hits a branch point of $t(s)$ on the $s$-surface, the corresponding steepest descent path hits a saddle point in the $t$-plane.

If we rotate $\theta$ in the positive direction, the steepest descent path $\mathscr{P}^{(n)} (\theta;\alpha_n)$ runs into a saddle point $t^{(m)}$ when $\theta=\theta_{nm}^+$. Likewise, if we rotate $\theta$ in the negative direction, the steepest descent path $\mathscr{P}^{(n)} (\theta;\alpha_n)$ hits a saddle $t^{(m)}$ when $\theta=\theta_{nm}^-$. By definition, the domain $\Delta^{(n)}$ is the union $\bigcup\nolimits_{\theta \neq \theta_{nm}^\pm} \mathscr{P}^{(n)} (\theta;\alpha_n)$, which is precisely the image of the points on the $s$-surface that can be seen from the branch point at the origin minus half lines with phases $2\pi\alpha_n-\theta_{nm}^\pm$ issuing from the points $s(t^{(m)})$ under the map $t(s)$. The boundary of the domain $\Delta^{(n)}$ is therefore consists of the images of these half lines under the map $t(s)$. It is easy to see that the image of the half line with phase $2\pi\alpha_n-\theta_{nm}^+$ emerging from $s(t^{(m)})$ under the map $t(s)$ is precisely the steepest descent path $\mathscr{P}^{(m)} (\theta _{nm}^ +  ,\alpha _{nm}^ +  )$ emanating from the adjacent saddle $t^{(m)}$. Similarly, the image of the half line with phase $2\pi\alpha_n-\theta_{nm}^-$ emerging from $s(t^{(m)})$ under the map $t(s)$ is the steepest descent path $\mathscr{P}^{(m)} (\theta _{nm}^ -  ,\alpha _{nm}^ - )$ emanating from the adjacent saddle $t^{(m)}$. In order to make the orientation of the domain $\Delta^{(n)}$ positive, the orientation of the steepest path $\mathscr{P}^{(m)} (\theta _{nm}^ -  ,\alpha _{nm}^ - )$ has to be reversed.

\section*{Acknowledgments}

The authors thank the referees for very helpful comments and suggestions for improving the presentation.

\end{document}